\documentclass[10pt]{amsart}
\usepackage{amssymb,amsmath,amsthm}
\usepackage{mathrsfs,dsfont,a4wide}
\usepackage{bm}
\usepackage{amssymb}
\usepackage{graphicx}
\usepackage{amsmath}
\usepackage[colorlinks=true]{hyperref}     
\usepackage{manfnt}
\usepackage{xcolor}
\hypersetup{linkcolor=black, citecolor=black}
\newtheorem{theorem} {Theorem}[section]
\newtheorem{lemma}[theorem]{Lemma}

\newtheorem{proposition}[theorem]{Proposition}

\newtheorem{remark}[theorem]{Remark}

\newtheorem{definition}[theorem]{Definition}

\numberwithin{equation}{section}

\newcommand{\tr}{\operatorname{tr}}
\newcommand{\dd}{\operatorname{d}\!}

\newcommand{\diver}{\operatorname{div}}
\newcommand{\curl}{\operatorname{curl}}

\newcommand{\n}{\bm{n}}
\newcommand{\e}{\bm{e}}

\newcommand{\normal}{\bm{\nu}}

\newcommand{\free}{\mathscr{F}}
\newcommand{\freetg}{\mathscr{F}_{tg}}

\newcommand{\p}{\bm{p}}

\newcommand{\vv}{\bm{v}}

\newcommand{\uu}{\bm u}

\newcommand{\As}{{\mathcal A}}
\newcommand{\Astg}{{\mathcal A}^{\rm tg}}
\newcommand{\hs}{{\mathcal H}}

\newcommand{\leb}{{\mathcal L}}

\newcommand{\ms}{{\mathcal M}}
\newcommand{\ns}{{\mathcal N}}

\newcommand{\Ns}{{\mathcal N}}
\newcommand{\Nstg}{{\mathcal N}_{tg}}

\newcommand{\hd}{{\mathcal{H}}^{d-1}}
\newcommand{\hdue}{{\mathcal{H}}^2}

\newcommand{\Om}{\Omega}
\newcommand{\Omb}{\overline{\Omega}}

\newcommand{\eps}{\varepsilon}
\newcommand{\weak}{\rightharpoonup}

\newcommand{\weakst}{\stackrel{\ast}{\rightharpoonup}}

\newcommand{\tsub}{\;\widetilde{\subseteq}\;}
\newcommand{\teq}{\;\widetilde{=}\;}

\definecolor{verde}{RGB}{20,150,100}

\newcommand{\R}{{\mathbb R}}

\newcommand{\N}{{\mathbb N}}

\newcommand{\Mtre}{M^{3\times 3}}

\begin{document}
\title{A shape optimization problem for nematic and cholesteric liquid crystal drops}

\author[A. Giacomini]
{Alessandro Giacomini}
\address[Alessandro Giacomini]{DICATAM, Sezione di Matematica, Universit\`a degli Studi di Brescia, Via Branze 43, 25133 Brescia, Italy}
\email[A. Giacomini]{alessandro.giacomini@unibs.it}

\author[S. Paparini]{Silvia Paparini}
\address[Silvia Paparini]{Dipartimento di Matematica, Università di Pavia, Via Ferrata 5, 27100 Pavia, Italy}
\email[S. Paparini]{silvia.paparini@unipv.it}

\begin{abstract}
We generalize the shape optimization problem for the existence of stable equilibrium configurations of nematic and cholesteric liquid crystal drops surrounded by an isotropic solution to include a broader family of admissible domains with inner boundaries, allowing discontinuities in the director field across them.
Within this setting, we prove the existence of optimal configurations under a volume constraint and show that the minimization problem is a natural generalization of that posed for regular domains.
\end{abstract}
\maketitle
\tableofcontents

\section{Introduction}
\label{sec:intr}
Liquid crystal phases are states of matter intermediate between crystalline solids and isotropic fluids, categorized as thermotropic or lyotropic based on whether temperature or concentration, respectively, drives their formation. These phases are birefringent like crystals and flow like liquids. There are three main liquid crystal phases, nematics, cholesterics and smectics. In the nematic phase the molecules have orientational order but no positional order, so that the mean orientation of the long axis of the molecules (or molecular assemblies) at the point $x$ can be represented by a unit vector $\n = \n(x)$ called the director. In the cholesteric (or chiral nematic) phase the molecules form a helical structure with an axis perpendicular to the local director. The smectic phases have orientational and some positional order. We restrict attention to only nematic and cholesteric phases.
\par
The classical elastic theory of liquid crystals goes back to the pioneering works of Oseen \cite{oseen:theory} and Frank \cite{frank:theory}. This theory is variational in nature, and features the director $\n$ as the only mesoscopic descriptor of local molecular order. The spatial distortion of a director field $\n$ is measured by its gradient $\nabla\n$.
Here we study a free-boundary problem, where a given quantity of nematic liquid crystal occupying a region $\Om\subset\mathbb{R}^3$ of volume $m>0$ (we treat liquid crystals as incompressible fluids) is surrounded by an isotropic fluid (which could well be its own melt) and can take on any desired shape. 
The energetic behaviour is described by
\begin{equation}
\label{eq:free-intr}
\free[\Om,\n]=\int_\Om W(\n,\nabla \n)\,d x+\int_{\partial\Om}g(\n,\normal)\,d\hdue,
\end{equation}
where $\normal$ denotes the exterior normal, and $\hdue$ denotes the two-dimensional Hausdorff measure which coincides with the usual area measure for sufficiently regular sets. 
\par
The volume integral, the {\it elastic free energy} of the configuration, measures the distortional cost produced by a deviation from a natural state, that is the director configuration (uniform for nematics, helicoidal for cholesterics) into which the liquid crystal naturally relaxes when neither external actions nor anchoring conditions affect its orientation.
The free-energy density $W(\n,\nabla\n)$ is chosen to be the most general frame-indifferent, even function quadratic in $\nabla\n$,
\begin{multline}
\label{eq:fb-intr}
W(\n,\nabla\n)=\frac{1}{2}K_{11}\left(\diver\n\right)^2+\frac{1}{2}K_{22}\left(\n\cdot\curl\n+q_0\right)^2+ \frac{1}{2}K_{33}|\n\times\curl\n|^{2}\\
+ \frac{1}{2}K_{24}\left[\tr(\nabla\n)^{2}-(\diver\n)^{2}\right].
\end{multline}
Here $q_0\in\R$ is the natural twist (zero for nematics, different from zero for cholesterics), while $K_{11}$, $K_{22}$, $K_{33}$, and $K_{24}$ are elastic constants characteristic of the material. They are often referred to as the splay, twist, bend, and saddle-splay constants, respectively, by the features of the different orientation fields, each with a distortion energy proportional to a single term in Eq. \eqref{eq:fb-intr} (see, for example, Ch. 3 of Ref. \cite{virga:variational}).
Throughout this paper, we assume  that the elastic constants in \eqref{eq:fb-intr} satisfy Ericksen’s inequalities\footnote{When $q_0=0$, these inequalities ensure that the cost associated with a distortion from a natural state is never
negative.} \cite{ericksen:inequalities}:
\begin{equation}
\label{eq:eineq}
(K_{11}-K_{24})>0, \quad (K_{22}-K_{24})>0, \quad K_{33}>0, \quad K_{24}>0.
\end{equation}
If in \eqref{eq:fb-intr} we set $q_0=0$, $K_{11}=K_{22}=K_{33}=K_{24}=K$, then we obtain the one-constant approximation $K|\nabla \n|^2$ to the free-energy density formula.
\par
The surface integral in \eqref{eq:free-intr}, the {\it surface free energy}, models the interaction of the liquid crystal with the surrounding isotropic solution. There, an anisotropic
surface tension $g$, depending on the orientation of $\n$ relative
to the exterior normal $\normal$, comes into play. A frequent choice for $g$ is given by the Rapini-Papoular formula \cite{rapini:distortion}
\begin{equation}
\label{eq:gamma_a}
g(\n,\normal):=\gamma\big(1+\lambda(\n\cdot\normal)^2\big),
\end{equation}
where $\gamma>0$ is an \emph{isotropic} surface tension and $\lambda>-1$ is a dimensionless \emph{anchoring strength}. When $\lambda>0$ the surface energy favours a tangential orientation of the optic axis at the free surface of the drop, whereas when $-1<\lambda<0$ it favours an homeotropic (along $\normal$) alignment. 
\par
A key theme in the mathematical modelling of liquid crystals \cite{ball:mathematics,ball:liquid,bedford} is the importance of a notion of admissible geometries and proper associated function space setting for the description of defects, i.e. discontinuities in the nematic director field which may be concentrated in isolated points, lines or surfaces. The same energy functional may predict different behaviour according to the function space or the admissible geometry used, as they may allow the description of one kind of defect but not another. Within Frank's theory, the classical variational problem of determining the optimal configuration of a LC droplet consists in the minimization of the functional \eqref{eq:free-intr} over pairs $(\Omega,\n)$, where $\Om\subset\mathbb{R}^3$ is an open, bounded set with a Lipschitz boundary and volume $m$, and $\n\in W^{1,2}(\Omega,\mathbb{S}^2)$,
\begin{equation}
\label{eq:minfree-intr}
\inf_{\Om}\min_{\n\in W^{1,2}(\Omega,\mathbb{S}^2)}\left\{\free\left[\Omega,\n\right]\,:\, |\Omega|=m, \, \hbox{with } \Omega \hbox{ open, bounded, and Lipschitz}\right\},
\end{equation}
where $|\Om|$ stands for the Lebesgue measure of $\Om$. Under these assumptions on $\Om$, it is well known that the minimization of \eqref{eq:free-intr} among all admissible Sobolev director fields, which yields the total energy of the shape $\Omega$
$$
\free(\Om):=\min_{\n\in W^{1,2}(\Omega,\mathbb{S}^2)}\free\left[\Omega,\n\right],
$$
is well posed for every $\gamma>0$, $\lambda>-1$, and whenever $W(\n,\nabla\n)$ satisfies the Ericksen's inequalities \eqref{eq:eineq}. In the classical formulation, balls, tactoids and cubs are included into the admissible geometries, while only point defects in the director field have finite volume energy.
\par
The mathematical study of the shape optimization problem (1.5) presents significant challenges due to the simultaneous determination of both the shape and the director field, as well as the interplay between elastic and surface energies, which exhibit different scaling behaviors. Rigorous results concerning the existence and regularity properties of minimizers have been achieved under specific additional conditions on the domains, which include convexity (see \cite{Lin-Poon}) or restrictions on the geometry of the admissible drops (see \cite{Geng-Lin} and references therein).
\vskip10pt
Droplets of liquid crystals suspended in an immiscible fluid \cite{volovik:topological,urbanski:liquid} or emerging in biphasic coexistence \cite{oakes:growth,krishnamurthy:topological,li:colloidal} show a delicate balance
between bulk elasticity, surface tension, and surface anchoring, yielding rich morphology of droplet shapes and fascinating inner structures. At equilibrium, they usually assume a shape and a director field described by the classical problem, and preserve a cohesiveness because of the positiveness of the surface-tension coefficient. However, there are exceptions. For sufficiently low values of surface tension, the elasticity of
the nematic phase dominates, and could lead to spontaneous division of chiral liquid crystal droplets during a specific sequence of phase transitions during temperature decrease \cite{lavrentovich:division}, in order to minimize elastic distortions. Another phenomena occurs during the early stages of the isotropic-nematic phase transition in large nematic liquid crystal droplets. Here, the nucleation, growth, and coalescence of nematic domains, characterized by different director orientations, can lead to the formation of disclinations of varying strengths or defect-walls, as described in \cite{kim:morphogenesis,lavrentovich:topological}. The energy barrier for coalescence is defined mainly by the elastic constants of the liquid crystal and the surface tension. 
\par
In view of these scenarios, limiting the analysis to the geometries and director fields described by the classical problem is restrictive. Therefore, it becomes necessary to extend the shape optimization problem to a broader class of admissible domains that admit {\it inner boundaries}, which may be lines (disclination defects) or surfaces (planar or wall-defects), and across which the director $\n$ can be {\it discontinuous}.
These internal diffusive boundaries, with a surface energy supported on them, are thus interpreted as a result of the percolation of the ambient fluid within the drop, or as holes that have collapsed into an inner surface. 
\par
This paper aims to generalize the notion of admissible geometries and space of directors from \eqref{eq:minfree-intr} to account for inner boundaries and the associated discontinuities of $\n$ across them.
Specifically, the generalized shape optimization problem fits within the general framework of {\it free discontinuity problems} for {\it special functions of bounded variations} $SBV$, as suggested in \cite[Section 4.6.4]{AFP}. SBV director fields are necessary to give finite energy to the defects, which are here concentrated on the inner boundaries across which $\n$ jumps. Weak anchoring is included in the model; as a result, the director is not constrained to a specific orientation on the boundary, thereby allowing for a further reduction of its distortion energy.
We show that the extension of the minimum problem \eqref{eq:minfree-intr} to generalized configurations is well posed and that minimizers can be energetically approximated through regular classical configurations. This indicates that the proposed variational problem is a natural generalization of the classical problem on regular domains.  
We refer the reader to \cite{bedford} for further applications of free discontinuity techniques to address issues in liquid crystals problems when $\Omega$ is fixed, or to \cite{BCS} for describing surface defects in smectic A thin films.
More precisely, we consider the family of admissible shapes $\As(\R^3)$, which consists of pairs $(\Om,\Gamma)$ where $\Om\subset \R^3$ is a set of {\it finite perimeter} and finite volume, and $\Gamma$ is a $\hdue$-countably rectifiable set contained in the set of points of density $1$ of $\Om$ (see Section \ref{sec:prel} for the precise definition). 
From a geometrical point of view, we interpret $(\Omega,\Gamma)$ as the set $\Om$ with an ``inner boundary'' $\Gamma$; thus, the shape is conceptually represented by the potentially irregular set $\Om\setminus \Gamma$. If $\Om$ is open, bounded and has a Lipschitz boundary, then $(\Om,\emptyset)\in \As(\R^3)$, indicating that the family $\As(\R^3)$ includes all regular domains. Given $(\Omega,\Gamma)\in\mathcal{A}(\R^3)$, we define the associated space of admissible directors by
\begin{multline*}
\Ns(\Omega,\Gamma):=\Bigg\{\n\in SBV(\R^3;\R^3): \n=0\ \text{ a.e. in $\Omega^c$}, |\n|=1 \text{ a.e. in }\Om, \\
\nabla \n\in L^2(\R^3;\Mtre), J_{\n}\tsub\partial^*\Omega\cup\Gamma\Bigg\},
\end{multline*}
where $\partial^*\Om$ stands for the reduced boundary of $\Om$ and $\tsub$ denotes inclusion up to $\hdue$-negligible sets.
 It turns out that $\ns(\Om,\Gamma)$ can be seen as a subset of a Hilbert space $\hs(\Om,\Gamma)$ which provides a generalization to the possibly irregular geometry $(\Om,\Gamma)$ of the standard Sobolev space $W^{1,2}$, while preserving key properties 
related to immersion and traces (c.f. Proposition \ref{prop:hs}): in this way $\Ns(\Omega,\Gamma)$ is the analogue of the space of Sobolev functions with values in $\mathbb S^2$.
\par
The natural extension of the total energy in \eqref{eq:free-intr} for $(\Om,\Gamma)\in \As(\R^3)$ and $\n\in \Ns(\Om,\Gamma)$ reads as

\begin{equation}
\label{eq:Es-intr}
\tilde\free\left[(\Omega,\Gamma),\n\right]:=\int_\Om W(\n,\nabla\n)\,dx+\int_{\partial^*\Om}g(\n,\normal)\,d\hdue+\int_{\Gamma}\left[g(\n^+,\normal)+g(\n^-,\normal)\right]\dd\hdue,
\end{equation}
where $\n^\pm$ denotes the traces of $\n$ on the rectifiable set $\Gamma$ associated to the orientation given by the normal $\normal$ (see Remark \ref{rem:traceG}). The surface energy on $\Gamma$ separately accounts for the two contributions of the traces $\n^\pm$, reflecting the interpretation of $\Gamma$ as an ``internal boundary'' resulting from an inner fold of the external boundary or the collapse of a hole.
Notice that the total energy $\tilde\free$ in \eqref{eq:Es-intr} reduces to the classical energy $\free$ in \eqref{eq:free-intr} when the geometry is regular. Specifically, if $A$ is an open bounded set with Lipschitz boundary, the family of directors $\Ns(A,\emptyset)$ essentially reduces to $W^{1,2}(A;\mathbb{S}^{2})$ (meaning that if $\n\in \ns(A,\emptyset)$, then $\n_{|A}\in W^{1,2}(A;\mathbb{S}^{2})$). Therefore, $\tilde\free[(A,\emptyset),\n]=\free[A,\n_{|A}]$. 
We will show in Theorem \ref{th:th_opt_director} that, as in the regular setting, the minimization of the energy \eqref{eq:Es-intr} on $\Ns(\Om,\Gamma)$ is well posed, yielding the \emph{total energy} $\tilde\free(\Om,\Gamma)$ \emph{of the shape} $(\Om,\Gamma)$,
\begin{equation}
\label{eq:Etot-intr}
\tilde\free(\Om,\Gamma):=\min_{\n\in\Ns(\Om,\Gamma)}\tilde\free\left[(\Om,\Gamma),\n\right].
\end{equation}
\par
The main result of this paper (Theorem \ref{thm:main2}) states that under the isoperimetric constraint that prescribes the volume $m>0$ of the domain $\Omega$, if the Ericksen's inequalities in \eqref{eq:eineq} hold and the constants of the surface energy satisfy $\gamma>0$ and $-\frac{1}{2}\le \lambda\le 1$, then the shape optimization problem
\begin{equation}
\label{eq:mainpb-intr}
\min\left\{ \tilde\free(\Om,\Gamma)\,:\, (\Om,\Gamma)\in \As(\R^3), |\Om|=m\right\}
\end{equation}
is well posed and its minimum value equals
$$
\inf\left\{ \free(A)\,:\,\text{$A$ has smooth boundary and $|A|=m$}\right\}
$$
so that in particular it is equal also to the infimum on Lipschitz domains.
A similar result holds when the droplet is confined within a convex (or alternatively smooth) open bounded box $D$ (see Theorem \ref{thm:main}).
\par
Existence of a solution for \eqref{eq:mainpb-intr} is obtained through the Direct Method of the Calculus of Variations. Specifically, if $(\Om_k,\Gamma_k)_{k\in\N}$ is a minimizing sequence, we obtain the bound
$$
\sup_k \{\hdue(\partial^*\Om_k)+\hdue(\Gamma_k)\}<+\infty.
$$
Consequently, the compactness for the sequence of sets $(\Om_k)_{k\in\N}$ follows from the theory of sets of finite perimeter, as a consequence of a vanishing/dichotomy/compactness alternative that exploits the translational invariance of the problem. On the other side, for the compactness of the inner interfaces $\Gamma_k$ we rely on a variational notion of convergence for rectifiable sets introduced in \cite{DMFT} for applications in fracture mechanics. This convergence enjoys compactness and lower semicontinuity properties very similar to those of Hausdorff convergence for compact connected sets with finite length. The restriction $-\frac{1}{2}\le \lambda\le 1$ on the surface energy constant $\lambda$ guarantees the lower semicontinuity of the surface energy as the geometries $(\Om,\Gamma)$ vary (see Lemma \ref{lem:glambda} and Theorem \ref{thm:lsc}). The density result for the energy of smooth configurations requires approximations similar to those exploited in \cite{braides:relaxation}: in our setting, additional difficulties arise because the maps are $\mathbb S^2$-valued and the surface energies also depend on the normal trace values. Further lower semicontinuity arguments can address the case of {\it tangential boundary conditions}, where $\n\cdot \normal=0$ and $\n^\pm\cdot \normal=0$ on $\partial^*\Om$ and $\Gamma$, respectively; the existence of stable equilibrium shapes is recovered also in this scenario (Theorem \ref{thm:maintg}).
\par
The paper is organized as follows. In Section \ref{sec:prel} we introduce the notation and recall the basic properties of sets of finite perimeter, rectifiable sets and $SBV$ functions employed throughout the paper. Generalized geometries and associated spaces of directors are defined in Section \ref{sec:geom}, while the main result concerning the existence of stable shapes is stated in Section \ref{sec:main}. In Section \ref{sec:tech} we collect the technical results concerning compactness, lower semicontinuity and approximation through smooth configurations, essential for the proof of the main result contained in Section \ref{sec:mainpf}. Finally, Section \ref{sec:tan} is devoted to the study of tangential boundary conditions.

\section{Notation and preliminaries}
\label{sec:prel}
If $E \subseteq \R^d$, we will denote with $|E|$ its $d$-dimensional Lebesgue measure, and by $\hd(E)$ its $(d-1)$-dimensional Hausdorff measure: we refer to \cite[Chapter 2]{EvansGariepy} for a precise definition, recalling that for sufficiently regular sets $\hd$ coincides with the usual area measure. Moreover, we denote by $E^c$ the complementary set of $E$, and by $1_E$ its characteristic function, i.e., $1_E(x)=1$ if $x \in E$, $1_E(x)=0$ otherwise. Finally, for $\alpha\in [0,1]$ we will write $E^\alpha$ for the points of density $\alpha$ for $E$ (see \cite[Definition 3.60]{AFP}).
\par
If $A \subseteq \R^d$ is open and $1 \le p \le +\infty$, we denote by $L^p(A)$ the usual space of $p$-summable functions on $A$ with norm indicated by $\|\cdot\|_p$. $W^{k,2}(A)$ will stand for the Sobolev  space of functions in $L^2(A)$ whose derivatives up to order $k$ in the sense of distributions belongs to $L^2(A)$. Finally $\ms_b(A;\R^d)$ will denote the space of $\R^d$-valued Radon measures on $A$, which can be identified with the dual of $\R^d$-valued continuous functions on $A$ vanishing at the boundary. We will denote by $|\cdot|$ its total variation.
\par
We say that $\Gamma\subseteq \R^d$ is $\hd$-countably rectifiable if
$$
\Gamma=N\cup \bigcup_{i\in \N}\Gamma_i
$$
where $\hd(N)=0$ and $\Gamma_i\subseteq \ms_i$ are Borel sets, where $\ms_i$ is a $C^1$-hypersurface of $\R^d$. It is not restrictive to assume that the sets $\Gamma_i$ are mutually disjoint. In the rest of the paper, we will write simply rectifiable in place of $\hd$-countably rectifiable. If $\Gamma_1,\Gamma_2$ are rectifiable, we will write $\Gamma_1\tsub \Gamma_2$ if $\hd(\Gamma_1\setminus \Gamma_2)=0$, and $\Gamma_1\teq \Gamma_2$ if $\hd((\Gamma_1\setminus \Gamma_2)\cup (\Gamma_2\setminus \Gamma_1))=0$.

\subsection{Functions of bounded variation}
\label{sec:bv}
Let $A\subseteq \R^d$ be an open set. We say that $u \in BV(A)$ if $u \in L^1(A)$ and its derivative in the sense of distributions is a finite Radon measure on $A$, i.e., $Du \in \ms_b(A;\R^d)$. $BV(A)$ is called the space of {\it functions of bounded variation} on $A$. $BV(A)$ is a Banach space under the norm $\|u\|_{BV(A)}:=\|u\|_{L^1(A)}+\|Du\|_{\ms_b(A;\R^d)}$.  We call $|Du|(A):=\|Du\|_{\ms_b(A;\R^d)}$ the {\it total variation} of $u$. We refer the reader to \cite{AFP} for
an exhaustive treatment of the space $BV$.
\par
If $u\in BV(A)$, then the measure $Du$ can be decomposed canonically (and uniquely) as
$$
Du=D^au+D^ju+D^cu.
$$
The measure $D^au$ is the absolutely continuous part (with respect to the Lebesgue measure) of the derivative: the associated density is denoted by $\nabla u\in L^1(A;\R^d)$. The measure $D^ju$ is the {\it jump part} of the derivative and it turns out that
$$
D^ju=(u^+-u^-)\otimes \nu_u \,\hd\lfloor J_u.
$$
Here $J_u$ is the {\it jump set} of $u$, $\nu_u$ is the normal to $J_u$, while $u^\pm$ are the upper and lower approximate limits at $x$. It turns out that $J_u$ is a  rectifiable set: if we choose the orientation given by a normal vector field $\nu_u$ we have $\hd$-a.e. 
$$
u^+=\gamma_r(u)\qquad\text{and}\qquad u^-=\gamma_l(u)
$$
where $\gamma_r(u)$ and $\gamma_l(u)$ are the right and left traces of $u$ on the rectifiable set $J_u$, associated to the orientation given by $\nu_u$.
Finally $D^cu$ is called the {\it Cantor part} of the derivative, and it vanishes on sets which are $\sigma$-finite with respect to $\hd$. Clearly $D^ju+D^cu$ is the singular part $D^su$ of $Du$ with respect to $\leb^d$.
\par
The space $SBV(A)$ of Special Functions of Bounded Variation on $A$ is defined as
$$
SBV(A):=\{u\in BV(A)\,:\, D^cu=0\},
$$
i.e., it is composed of those functions of bounded variation with vanishing Cantor part.

\subsection{Sets of finite perimeter}
\label{sec:per}
Given $E\subseteq \R^d$ measurable and $A\subseteq \R^d$ open, we say that $E$ has finite perimeter in $A$ (or simply has finite perimeter if $A=\R^d$) if
$$
Per(E;A):=\sup\left\{\int_E div(\varphi)\,dx\,:\, \varphi\in C^\infty_c(A;\R^d), \|\varphi\|_\infty\le 1\right\}<+\infty.
$$
If $|E|<+\infty$, then $E$ has finite perimeter if and only if $1_E\in BV(\R^d)$.
It turns out that
$$
D1_E=\nu_E \hd\lfloor \partial^*E,\qquad Per(E;\R^d)=\hd(\partial^*E),
$$
where the rectifiable set $\partial^*E$ is called the {\it reduced boundary} of $E$, and $\nu_E$ is the associated inner approximate normal (see \cite[Section 3.5]{AFP}). It turns out that $\partial^*E\subseteq \partial E$, but the topological boundary can in general be much larger than the reduced one.

\par
The next proposition is the collection of two approximation results proved in \cite{CoTo}, and will prove particularly useful for our problem.

\begin{proposition}[\bf Interior approximation via smooth sets]\label{pro:comitorres}
Let $\mu$ be a Radon measure on $\R^d$ such that $\mu<<\hd$ and let $E\subset\R^d$ be a bounded set of finite perimeter. Let $u_k:=1_E*\rho_{\varepsilon_k}$, where $\rho_k$ is a regularizing kernel, and let $A_{k,t}:=\{u_k>t\}$. Then for a.e. $t\in (0,1)$, $A_{k,t}$ is a smooth set and, for a.e. $t\in (1/2,1)$, the sequence $(A_{k,t})_k$ provides an interior approximation of $E$, i.e., $\hd(\partial A_{k,t})\to \hd(\partial^*E)$, 
$$
\lim_{k\to+\infty}|\mu|(A_{k,t}\Delta E^1)=0
$$
and
$$
\lim_{k\to+\infty}\hd(\partial A_{k,t}\setminus E^1)=0.
$$
\end{proposition}

\subsection{A variational convergence for rectifiable sets}
\label{sec:sigma2} 
 We recall the notion of $\sigma^2$-convergence for rectifiable sets introduced in \cite{DMFT} to deal with problems in fracture mechanics. It is a variational notion of convergence for rectifiable sets which enjoy compactness and lower semicontinuity properties under uniform bound for the associated $\hd$ measure, which are very similar to that enjoyed by connected closed sets in $\R^2$ with respect to Hausdorff convergence in view of Go\c l\" ab semicontinuity theorem. The definition is based on the use of the space $SBV$. Recall the notation $\Gamma_1\tsub \Gamma_2$ and $\Gamma_1\teq \Gamma_2$ which denote inclusion and equality up to $\hd$-negligible sets.

\begin{definition}[\bf $\sigma^2$-convergence]
\label{def:sigma2}
Let $D\subset\R^d$ be open and bounded, and let $\Sigma_n,\Sigma\subset D$ be rectifiable sets such that $\hd(\Sigma_n),\hd(\Sigma)\le C$ for some $C>0$. We say that 
$$
\Sigma_n \to \Sigma \qquad\text{in the sense of $\sigma^2$-convergence}
$$ 
if the following two conditions are satisfied.
\begin{itemize}
\item[(a)] If $u_j\in SBV(D)$ with $J_{u_j}\tsub \Sigma_{n_j}$ for some sequence $n_j\to+\infty$, and $u\in SBV(D)$ are such that $\|u_j\|_\infty, \|u\|_\infty \le C$,
$$
u_j \to u\qquad\text{strongly in }L^1(D)
$$
and
$$
\nabla u_j \to \nabla u\qquad\text{weakly in }L^2(D;\R^d),
$$
then $J_{u}\tsub \Sigma$.

\item[(b)] There exist a function $u\in SBV(D)$ with $\nabla u\in L^2(D;\R^d)$ and a sequence $u_n \in SBV(D)$ with $\|u\|_\infty, \|u_n\|_\infty \le C$,
$$
u_n \to u\qquad\text{strongly in }L^1(D)
$$
and
$$
\nabla u_n \to \nabla u\qquad\text{weakly in }L^2(D;\R^d),
$$
such that  $J_u \teq\Sigma$ and $J_{u_n}\tsub \Sigma_{n}$ for every $n\in\N$.
\end{itemize}
\end{definition}

\begin{remark}
{\rm
Condition (a) guarantees that $\Sigma$ contains the jump sets of functions which are suitable limits of functions jumping on $\Sigma_n$. Condition (b) ensures that $\Sigma$ is the smallest set which enjoys this property. The notion of convergence introduced in \cite{DMFT} can indeed be generalized to an exponent $p\in ]1,+\infty[$: we will use only the case $p=2$.
}
\end{remark}

The following compactness and lower semicontinuity result holds true, and will be fundamental for our analysis.

\begin{theorem}[\bf Compactness and lower semicontinuity for $\sigma^2$-convergence]
\label{thm:sigma2}
Let $D\subset\R^d$ be open and bounded. For every sequence $\Sigma_n\subset D$ of rectifiable sets such that $\hd(\Sigma_n)\le C$, there exist a rectifiable set $\Sigma\subset D$ and a subsequence $\Sigma_{n_k}$ such that
$$
\Sigma_{n_k} \to \Sigma \qquad\text{in the sense of $\sigma^2$-convergence.}
$$
Moreover we have
$$
\hd(\Sigma) \le \liminf_{n\to +\infty} \hd(\Sigma_n).
$$
\end{theorem}

\section{Generalized admissible configurations and associated energies}
\label{sec:geom}
In this section we introduce the formal definition of the family of admissible configurations and of the associated space of directors.
%

\subsection{Family of admissible geometries}
The class of admissible geometries we will consider is as follows. Recall that $\Om^1$ denotes the set of points of density one for $\Om$.

\begin{definition}[{\bf Generalized admissible configurations}]
\label{eq:admissible_geometries}
We say that the pair $(\Omega,\Gamma)$ is an admissible geometry, and we write $(\Omega,\Gamma)\in\mathcal{A}(\R^3)$, if $\Omega\subset\R^3$ is a set of finite perimeter with $|\Om|<+\infty$, and $\Gamma\subset \Omega^1$ is a rectifiable set with $\hdue(\Gamma)<+\infty$. 
\end{definition}

\begin{remark}
{\rm
From a geometrical point of view, $(\Om,\Gamma)\in\As(\R^3)$ can be thought of as the set $\Om$ with an "inner" crack $\Gamma$. Consequently, the domain is roughly given by the set $\Om\setminus \Gamma$, which may be irregular. If $\Om$ is open, bounded, and has a Lipschitz boundary, then $(\Om,\emptyset)\in \As(\R^3)$, meaning that the family $\As(\R^3)$ essentially includes all regular domains.
}
\end{remark}

\subsection{Family of admissible directors}
To associate suitable generalized bulk and surface energies with an admissible geometry $(\Om,\Gamma)\in\As(\R^3)$, we need to define a functional space that serves as an analogue to the Sobolev space $W^{1,2}$ on the potentially irregular set $\Om\setminus \Gamma$. 
\par
A natural candidate to replace the Sobolev space $W^{1,2}$ is given by
\begin{multline*}
\hs(\Om,\Gamma)
:=\Bigg\{\uu\in SBV(\R^3;\R^3): \uu=0\ \text{ a.e. in $\Omega^c$}, 
\nabla \uu\in L^2(\R^3;\Mtre), \\
J_{\uu}\tsub\partial^*\Omega\cup\Gamma, 
\int_{\partial^*\Om \cup \Gamma}[|\uu^+|^2+|\uu^-|^2]\,d\hdue<+\infty\Bigg\},
\end{multline*}
endowed with the scalar product
\begin{equation}
\label{eq:scalarpd}
(\uu,\vv)_{\hs}:=\int_\Om \nabla \uu:\nabla \vv\,dx+\int_{\partial^*\Om \cup \Gamma} [\uu^+\cdot \vv^++\uu^-\cdot \vv^-]\,d\hdue.
\end{equation}

\begin{remark}
\label{rem:traceG}
{\rm
The values $\uu^\pm$ involved in the definition of $\hs(\Om,\Gamma)$ are the trace values of $\uu$ on the rectifiable set $\partial^*\Om \cup \Gamma$ in the following sense.
\par
Given a rectifiable set $K\subset \R^3$ with $\hdue(K)<+\infty$, we can, by definition, write
\begin{equation*}
K=E\cup  \bigcup_{i=0}^\infty K_i,
\end{equation*}
where $\hdue(E)=0$, and for each $i\in \N$ the Borel sets $K_i$ are subsets of a $\mathcal{C}^1$-manifold $\ms_i$ of dimension $2$, with $K_i\cap K_j=\emptyset$ for $i\neq j$. 
\par
It is not restrictive, up to reducing $K_i$, to assume that each $\ms_i$ is orientable, with an associated normal vector filed $\normal_i$. Consequently, two continuous trace operators can be defined from $BV(\R^3;\R^3)$ to $L^1(\ms_i;\R^3)$, corresponding to the ``left'' and ``right'' traces. For any $\uu\in BV(\R^3;\R^3)$, let $\gamma_+^i(\uu)$ and $\gamma_-^i(\uu)$ denote the traces of $\uu$ on $K_i$
relative to $\normal_i$ and $-\normal_i$, respectively. Thus, we have
$$
\uu^+=\sum_i \gamma_+^i(\uu)\qquad\text{and}\qquad \uu^-=\sum_i \gamma_-^i(\uu).
$$
}
\end{remark}

The family of admissible directors that we consider is the following.

\begin{definition}[\bf Admissible directors]
Let $(\Omega,\Gamma)\in\mathcal{A}(\R^3)$. We set
\begin{multline}
\label{eq:Ns}
\Ns(\Omega,\Gamma):=\Bigg\{\n\in SBV(\R^3;\R^3): \n=0\, \text{ a.e. in $\Omega^c$}, |\n|=1 \text{a.e. in }\Om, \\
\nabla \n\in L^2(\R^3;\Mtre), J_{\n}\tsub\partial^*\Omega\cup\Gamma\Bigg\}.
\end{multline}
\end{definition}

\begin{remark}
{\rm
Notice that if $\Om$ is an open, bounded set with a Lipschitz boundary, then $(\Om,\emptyset)\in \As(\R^3)$. In this case, the spaces $\hs(\Om,\emptyset)$ and $\Ns(\Om,\emptyset)$ are simply given by extending functions in the Sobolev spaces $W^{1,2}(\Om;\R^3)$ and $W^{1,2}(\Om;\mathbb{S}^{2})$ to zero outside $\Om$, respectively.
}
\end{remark}

We collect in the proposition below key properties of the spaces $\hs(\Om,\Gamma)$ and $\ns(\Om,\Gamma)$. The proof is provided in Section \ref{subsec:basic}.

\begin{proposition}
\label{prop:hs}
The following properties hold true.
\begin{itemize}
\item[(a)] The scalar product defined in \eqref{eq:scalarpd} endows $\hs(\Om,\Gamma)$ with a Hilbert space structure.
\item[(b)] There exists a constant $C>0$ (depending on $|\Om|$ and $\hdue(\partial^*\Om\cup\Gamma)$) such that for every $\uu\in \hs(\Om,\Gamma)$
\begin{equation}
\label{eq:bvbound}
\|\uu\|_{BV}\le C\|\uu\|_\hs.
\end{equation}
\item[(c)] The immersion of $\hs(\Om,\Gamma)$ into $L^1(\R^3;\R^3)$ is compact.
\item[(d)] $\ns(\Om,\Gamma)$ is sequentially weakly closed in $\hs(\Om,\Gamma)$. Moreover, if $(\n_k)_{k\in\N}$ is a sequence in $\ns(\Om,\Gamma)$ such that $\n_k\weak \n$ weakly in $\hs(\Om,\Gamma)$, then
$$
\n_k^\pm \to \n^\pm \qquad\text{strongly in }L^2(\partial^*\Om\cup\Gamma;\R^3).
$$
\end{itemize}
\end{proposition}

\subsection{Energies for admissible configurations}
In this section, we extend the total energy $\free$ defined in \eqref{eq:free-intr} for regular domains and director fields to the generalized geometries and directors introduced in Sections 3.1 and 3.2. The bulk and surface free energies describing geometries $(\Om,\Gamma)\in \As(\R^3)$ and directors $\n\in \Ns(\Om,\Gamma)$ are defined, respectively, as
\begin{equation}
\label{eq:Eb}
\tilde\free_\mathrm{b}\left[(\Omega,\Gamma),\n\right]:=\int_\Om W(\n,\nabla\n)\,dx
\end{equation}
and
\begin{equation}
\label{eq:Es}
\tilde\free_\mathrm{s}\left[(\Omega,\Gamma),\n\right]:=\int_{\partial^*\Om}g(\n,\normal)\,d\hdue+\int_{\Gamma}\left[g(\n^+,\normal)+g(\n^-,\normal)\right]\dd\hdue.
\end{equation}
The bulk free-energy $W$ is defined in \eqref{eq:fb-intr}, and satisfy Ericksen's inequalities \eqref{eq:eineq}. 
Concerning the surface energy \eqref{eq:Es}, $\normal$ denotes a unit normal vector field defined on the rectifiable set $\partial^*\Om \cup \Gamma$, $\n^\pm$ are the associated traces of $\n$ on $\Gamma$, and $g$ is the surface energy density given by \eqref{eq:gamma_a},
with $\gamma>0$ and $\lambda>-1$.


\begin{remark}
\label{rem:reg-reduce}
{\rm
As noticed in the introduction, $\tilde\free_\mathrm{b}$ and $\tilde\free_\mathrm{s}$ naturally extend the bulk and surface energies defined in \eqref{eq:free-intr} for regular configurations to describe both the geometries of $\mathcal{A}(\R^3)$ and the associated director fields in $\ns(\Om,\Gamma)$.   
}
\end{remark}

\begin{remark}[\bf Properties of the bulk energy density]
\label{rem:bulk}
{\rm
Under the assumption of validity of Ericksen's inequalities \eqref{eq:eineq}, it is well known that the bulk energy density $W:\R^3\times M^{3\times 3}\to [0,+\infty[$ satisfies the following properties:
\begin{itemize}
\item[(a)] There exist constants $c_1, c_2, c_3, c_4>0$ such that for every $\bm s\in \R^3$ and $\bm M\in M^{3\times 3}$
\begin{equation}
\label{eq:lower_upper_W_chol}
c_1\min\{|\bm s|^2,1\}|\bm M|^2-c_2\leq W(\bm s,\bm M)\leq c_3 \max\{1,|\bm s|^2\} |\bm M|^2+c_4.
\end{equation}
\item[(b)] For every $\bm s\in\mathbb{R}^3$, the map $\bm M\mapsto W\left(\bm s,\bm M\right)$ is quasiconvex.
\end{itemize}
Quasiconvexity follows by observing that the first three terms in \eqref{eq:fb-intr} are convex in $\bm M$, while the last term is associated to a {\it null lagrangian} (see \cite[Lemma 1.2]{HKL}). Properties (a) and (b) also hold in the two dimensional setting.
}
\end{remark}

According to \eqref{eq:Etot-intr}, we define the \emph{total energy of the shape} $(\Om,\Gamma)$ as
\begin{equation}
\label{eq:Etot}
\tilde\free(\Om,\Gamma):=\min_{\n\in \Ns(\Om,\Gamma)}\left[\tilde\free_\mathrm{b}\left[(\Omega,\Gamma),\n\right]+\tilde\free_\mathrm{s}\left[(\Omega,\Gamma),\n\right]\right];
\end{equation}
The following theorem establishes the well posedness of this minimization problem.
\begin{theorem}
\label{th:th_opt_director}
Assume that Ericksen's inequalities \eqref{eq:eineq} hold, and let the constants of the surface energy satisfy $\gamma>0$ and $\lambda>-1$. 
Then, the minimum in \eqref{eq:Etot} is attained.
\end{theorem}

\begin{proof}
Let $(\n_k)_{k\in\N}$ be a minimizing sequence for the problem \eqref{eq:Etot} in $\ns(\Om;\Gamma)$. According to Remark \ref{rem:bulk}, we have
$$
\|\n_k\|_\hs^2=\int_\Om |\nabla \n_k|^2\,dx+\int_{\partial^*\Om\cup \Gamma}[|\n_k^+|^2+|\n_k^-|^2]\,d\hdue\le C \left\{\tilde\free[(\Om,\Gamma),\n_k]+\hdue(\partial^*\Om\cup \Gamma)+1\right\},
$$
where $C>0$ is independent of $k$. Consequently, $(\n_k)_{k\in\N}$ is bounded in $\hs(\Om,\Gamma)$. By Proposition \ref{prop:hs}, there exists $\n\in \ns(\Om,\Gamma)$ such that, up to a subsequence, the following holds:
$$
\nabla \n_k\weak \nabla \n\qquad\text{weakly in }L^2(\R^3;\Mtre)
$$
$$
\n_k^\pm \to \n^\pm\qquad\text{strongly in }L^2(\partial^*\Om\cup\Gamma;\R^3),
$$
and
$$
\n_k\to \n\qquad\text{strongly in }L^1(\R^3;\R^3).
$$
Thanks to point (b) in Remark \ref{rem:bulk} and by the lower semicontinuity result for quasiconvex bulk energies in SBV (c.f. \cite[Theorem 5.29]{AFP} applied to $W+\eps |\nabla \n|^2$ for sufficiently small $\eps$), we obtain
$$
\int_\Om W(\n,\nabla \n)\,dx\le \liminf_{k\to+\infty} \int_\Om W(\n_k,\nabla \n_k)\,dx.
$$
Furthermore, since traces $\n_k^\pm$ converge strongly in $L^2(\partial^*\Om\cup\Gamma;\R^3)$, it follows that
$$
\lim_{k\to+\infty}\int_{\partial^*\Om}g(\n_k,\normal)\,d\hdue=\int_{\partial^*\Om}g(\n,\normal)\,d\hdue
$$
and
$$
\lim_{k\to+\infty}
\int_{\Gamma}[g(\n^+_k,\normal)+g(\n^-_k,\normal)]\,d\hdue=
\int_{\Gamma}[g(\n^+,\normal)+g(\n^-,\normal)]\,d\hdue.
$$
Hence, $\n$ is a minimum point for $\tilde\free[(\Om,\Gamma),\cdot]$, thereby concluding the proof.
\end{proof}

\section{Stable configurations and the main result}
\label{sec:main}
Given $m>0$, we set
$$
\As_m(\R^3):=\{(\Om,\Gamma)\in \As(\R^3), |\Om|=m\}
$$
and
$$
\As^{reg}_m(\R^3):=\{\text{$A\subset\R^3$ open set with a smooth boundary, with $|A|=m$}\}.
$$


The main result of the paper can be stated as follows.

\begin{theorem}[\bf Existence of stable configurations]
\label{thm:main2}
Assume that Ericksen's inequalities \eqref{eq:eineq} hold, and let the constants of the surface energy satisfy $\gamma>0$ and $-\frac{1}{2}\le \lambda\le 1$. Then the minimum problem
\begin{equation}
\label{eq:mainpb2}
\min_{\As_m(\R^3)} \tilde\free
\end{equation}
is well posed and every minimizer $(\Om,\Gamma)$ is such that $\Om$ is bounded. Moreover,
\begin{equation}
\label{eq:reg-eq2}
\min_{\As_m(\R^3)} \tilde\free=\inf_{\As_{m}^{reg}(\R^3)} \free.
\end{equation}
\end{theorem}


A similar result (with a simpler proof) holds for configurations contained in a given convex (or alternatively smooth) bounded box $D$, which can be considered a {\it design region}. We refer the reader to Theorem \ref{thm:main}.

\section{Some technical results}
\label{sec:tech}
In this section, we collect several technical results that will lead us to the proof of our main theorem. While these results are generally applicable in $d\geq2$ dimension (generalizing the notions of admissible configurations and the associated space of directors to $\R^d$), we focus on the three-dimensional case, as it is specifically relevant to our application to liquid crystals.


\subsection{General properties of the space of directors}
\label{subsec:basic}

In this subsection, we provide the proof of Proposition \ref{prop:hs} concerning the space $\hs(\Om,\Gamma)$, where $(\Om,\Gamma)\in \As(\R^3)$.

\begin{proof}[Proof of Proposition \ref{prop:hs}]
We divide the proof in three steps.

\vskip10pt\noindent
{\bf Step 1.} 
The Sobolev embedding of $BV(\R^3;\R^3)$ into $L^{\frac{3}{2}}(\R^3;\R^3)$, combined with the Cauchy-Schwarz inequality, provides the following estimate for every $\uu\in \hs(\Om,\Gamma)$ (recall that $\Om$ has finite measure)
\begin{multline}
\label{eq:boundhs}
\|\uu\|_{L^{\frac{3}{2}}(\R^3;\R^3)}\le C_3 |D\uu|(\R^3)=C_3 \left[ \int_{\R^3} |\nabla \uu|\,dx+\int_{J_{\uu}}|\uu^+-\uu^-|\,d\hs^{2}\right]\\
\le C_3 \left[ \int_{\Om} |\nabla \uu|\,dx+\int_{\partial^*\Om \cup \Gamma}[|\uu^+|+|\uu^-|]\,d\hs^{2}\right]\le C\sqrt{(\uu,\uu)_{\hs}}.
\end{multline}
Here $C$ is a constant that depends on $|\Om|$ and $\hs^{2}(\partial^*\Om\cup\Gamma)$.
Since $\Om$ has finite measure, inequality \eqref{eq:boundhs} yields
\begin{multline*}
\|\uu\|_{BV}\le \|\nabla \uu\|_{1}
+\int_{\partial^*\Om\cup \Gamma}\left[|\uu^+|+|\uu^-|\right]\,d\hdue +\|\uu\|_{1} \\
\le C\left[ \|\nabla \uu\|_{2}+\|\uu^+\|_{2}+\|\uu^-\|_{2}\right]^{1/2}=C\|\uu\|_\hs,
\end{multline*}
which completes the proof of point (b). Point (c) follows directly from the local compact embedding of $BV$ into $L^1$.

\vskip10pt\noindent
{\bf Step 2.} The pairing \eqref{eq:scalarpd} is clearly a scalar product (recall \eqref{eq:boundhs}). To complete the proof of point (a), let us verify that Cauchy sequences in $\hs(\Om,\Gamma)$ are convergent. Let $(\uu_n)_{n\in\N}$ be a Cauchy sequence in $\hs(\Omega,\Gamma)$. Then, there exist 
$\bm\Phi\in L^2(\R^3;\Mtre)$, $\bm\alpha^+,\bm \alpha^-\in L^2(\partial^*\Om\cup\Gamma;\R^3)$ with $\Phi=0$ a.e. on $\Om^c$, such that
$$
\nabla \uu_n\to \bm\Phi\qquad\text{strongly in }L^2(\R^3;\Mtre),
$$
and
$$
\uu^\pm_n\to \bm \alpha^\pm  \qquad\text{strongly in }L^2(\partial^*\Om\cup\Gamma;\R^3).
$$
In view of \eqref{eq:bvbound}, we can write
$$
\|\uu_n-\uu_m\|_{BV}\le C\|\uu_n-\uu_m\|_\hs\to 0,
$$
which implies that there exists $\uu \in BV(\R^3;\R^3)$ such that
$$
\uu_n\to \uu\qquad\text{strongly in }BV(\R^3;\R^3).
$$
It follows immediately that $\uu=0$ a.e. on $\Om^c$, and $\bm\Phi=\nabla \uu$. Moreover, $\bm\alpha^\pm=\uu^\pm$ thanks to the continuity of the local trace operators that determines the values on $\partial^*\Om\cup \Gamma$ (see Remark \ref{rem:traceG}). Since $\uu \in SBV(\R^3;\R^3)$ ($SBV$ is a closed subset of $BV$ in the strong topology), we get that $\uu \in \hs(\Om,\Gamma)$ and that $\uu_n\to \uu$ in the associated norm, from which completeness follow.

\vskip10pt\noindent
{\bf Step 3.}
Concerning point (d), let $(\n_k)_{k\in\N}$ be a sequence in $\ns(\Om,\Gamma)$ and assume that $\n_k\weak \n$ weakly in $\hs(\Om,\Gamma)$. By point (c), it follows that 
$$
\n_k\to \n\qquad\text{strongly in }L^1(\R^3;\R^3).
$$
Consequently, $\n=0$ a.e. on $\Om^c$ and $|\n|=1$ a.e. on $\Om$, ensuring that $\n\in \ns(\Om,\Gamma)$.
\par
We now address the convergence of the traces. According to Remark \ref{rem:traceG}, the result follows from applying the following theorem concerning traces of $SBV$ functions to the local trace operators involved in characterizing the traces on $\partial^*\Om\cup\Gamma$.
\par
Let $A\subset \R^3$ be an open, bounded set with a Lipschitz boundary, and let $K\subset A$ be a rectifiable set with $\hdue(K)<+\infty$. Assume that $(u_n)_{n\in \N}$ is a sequence in $SBV(A)$ such that $J_{u_n}\tsub K$, $\|\nabla u_n\|_2+\|u_n\|_\infty\le C$ and
$u_n\to u$ strongly in $L^1(A)$ for some $u\in SBV(A)$.
Then
\begin{equation}
\label{eq:convL1gamma}
\gamma(u_n) \to \gamma(u)\qquad\text{strongly in }L^1(\partial A),
\end{equation}
where $\gamma$ is the trace operator from $BV(A)$ to $L^1(\partial A)$ (The convergence is indeed strong in every $L^p$ space, thanks to the uniform $L^\infty$-bound).
\par
To prove property \eqref{eq:convL1gamma} we proceed as follows. Note that also $\|\nabla u\|_2+\|u\|_\infty\le C$. Indeed, we have that $u_n\weakst u$ weakly* in $BV(A)$, with separate convergence for the absolutely continuous and jump parts of the derivative. Moreover, by the trace theory in $BV$, there exists $C_1>0$ such that for every $\eps>0$ and for every $u\in BV(A)$,
$$
\|\gamma(u)\|_{L^1(\partial A)}\le C_1\left[ |Du|(A_\eps)+\frac{1}{\eps}\|u\|_{L^1(A_\eps)}\right],
$$
where $A_\eps:=\{x\in A\,:\, dist(x,\partial A)<\eps\}$. We can then write
\begin{multline*}
\|\gamma(u_n)-\gamma(u)\|_{L^1(\partial A)}\\
\le C_1\left[ \int_{A_\eps} |\nabla u_n-\nabla u|\,dx+\int_{K\cap A_\eps} |(u_n^+-u^+)-(u_n^--u^-)|\,d\hdue+\frac{1}{\eps}\|u_n-u\|_{L^1(A)}\right]\\
\le C_2\left[ |A_\eps|^{1/2}+\hdue(K\cap A_\eps)+\frac{1}{\eps}\|u_n-u\|_{L^1(A)}\right].
\end{multline*}
Using the fact that $\eps$ is arbitrary, we get that \eqref{eq:convL1gamma} follows.
\end{proof}

\subsection{Compactness and lower semicontinuity results}
\label{sec:lsc}
Let us start with the following property concerning the surface energy density.

\begin{lemma}[\bf Convexity properties of the surface energy density]
\label{lem:glambda}
Let $-\frac{1}{2}\le \lambda\le 1$, and let $\hat g:\R^3\times \R^3 \to [0,+\infty]$ be given by
$$
\hat g(\bm v,\bm p):=
\begin{cases}
\gamma\left( |\bm p||\vv|^2+\lambda\frac{\left(\vv\cdot\bm p\right)^2}{|\bm p|}\right)&\text{if }\bm p\not=0\\
0 &\text{if }\bm p=0.
\end{cases}
$$
Then $\hat g$ is convex and positively one homogeneous in the second variable, with $\hat g=g$ on $\mathbb{S}^{2}\times \mathbb{S}^{2}$.
\end{lemma}

\begin{proof}
We can assume $\bm v\not=\bm 0$. Since for $\bm p\not=\bm 0$ we have
\begin{equation*}
\hat{g}(\vv,\bm p)=\gamma |\vv|^2\left(|\bm p|+\lambda\frac{\left(\frac{\vv}{|\vv|}\cdot\bm p\right)^2}{|\bm p|}\right),
\end{equation*}
it suffices to concentrate on the convexity of the function
\begin{equation}
\label{eq:extension_g}
\p\to|\bm p|+\lambda\frac{\left(\bm e\cdot\bm p\right)^2}{|\bm p|},
\end{equation}
where $\bm e:=\frac{\vv}{|\vv|}$.
We divide the proof into two steps.

\vskip10pt\noindent{\bf Step 1.} Let us show that the function $h:\R^2\to [0,+\infty[$ defined as
$$
h(\bm x)=h(x_1,x_2):=
\begin{cases}
|\bm x|+\lambda\frac{x_1^2}{|\bm x|},&\text{if }\bm x\not=0\\
0 &\text{if }\bm x=0
\end{cases}
$$
is convex. It suffices to check that outside the origin the hessian is positive semidefinite. A straightforward computation shows that 
$$
D^2h=\frac{(1-\lambda)x_1^2+(1+2\lambda)x_2^2}{|\bm x|^3}\bm B,\qquad\text{with } \bm B:=\left(\bm{\mathrm{I}}-\frac{\bm x\otimes\bm x}{|\bm x|^2}\right).
$$
The result follows since $\bm B$ is positive definite and $\lambda \in [-\frac{1}{2},1]$.

\vskip10pt\noindent{\bf Step 2.} Let $V_2$ be a two dimensional subspace of $\R^3$. If $V_2$ contains $\bm e$, then, in view of Step 1, the map \eqref{eq:extension_g} is convex on $V_2$. If $\bm e\not\in V_2$, let $\bm e_1$ denote the orthogonal projection of $\bm e$ onto $V_2$; for $\bm p\in V_2$ with $\bm p\not=0$  we can write
$$
|\bm p|+\lambda\frac{\left(\bm e\cdot\bm p\right)^2}{|\bm p|}=|\bm p|+\lambda |\bm e_1|^2\frac{\left(\frac{\bm e_1}{|\bm e_1|}\cdot\bm p\right)^2}{|\bm p|},
$$
which is again convex on $V_2$ by Step 1 since $\lambda |\bm e_1|^2\in [-1/2,1]$.
\end{proof}

\begin{remark}
\label{rem:surf-dim}
{\rm
The previous result holds in general dimension, in particular also in dimension two.
}
\end{remark}

The following compactness and lower semicontinuity result will be fundamental for our analysis.

\begin{theorem}
\label{thm:lsc}
Let $(\Omega_k,\Gamma_k)_{k\in\N}$ be a sequence in $\mathcal{A}(\R^3)$ such that 
\begin{equation}
\label{eq:boundProb}
\hdue(\partial^*\Om_k\cup\Gamma_k)\le C 
\end{equation}
for some positive constant $C>0$ independent of $k$. Assume that as $k\to +\infty$ 
\begin{equation}
\label{eq:s-omega_n}
1_{\Om_k}\to 1_\Om\qquad\text{strongly in }L^1(\R^3)
\end{equation}
for some set of finite perimeter $\Om\subset \R^3$. 
\par
Then there exists $\Gamma\subset \R^3$ rectifiable with $(\Omega,\Gamma)\in\mathcal{A}(\R^3)$ and such that
the following property holds true: if $(\n_k)_{k\in\N}$ is a sequence such that each $\n_k\in \ns(\Om_k,\Gamma_k)$  with $(\nabla \n_k)_{k\in\N}$ bounded in $L^2(\R^3;\Mtre)$, then there exists $\n\in \ns(\Om,\Gamma)$ such that, up to a subsequence,
$$
\n_k\to  \n\quad\text{strongly in $L^1(\R^3;\R^3)$},
$$
$$
\nabla \n_k\weak \nabla \n\quad\text{weakly in $L^2(\R^3;\Mtre)$},
$$
and
\begin{equation}
\label{eq:superficie}
\int_{\partial^*\Omega\cup\Gamma}[\varphi(\n^+,\normal)+\varphi(\n^-,\normal)]\:d\hdue\le\liminf_{k\to+\infty}\int_{\partial^*\Omega_k\cup\Gamma_k}[\varphi(\n^+_k,\normal_k)+\varphi(\n^-_k,\normal_k)]\:d\hdue,
\end{equation}
for every continuous function $\varphi=\varphi(\bm v,\bm p):\R^3\times \R^3\to [0,+\infty[$, even, convex and positively one homogeneous in $\bm p$.
\end{theorem}

\begin{proof}
We will make use of the notion of $\sigma^2$-convergence of rectifiable sets introduced in \cite{DMFT} (see Section \ref{sec:sigma2}). From \eqref{eq:boundProb} and Theorem \ref{thm:sigma2}, and employing a diagonal argument, it follows that there exists a rectifiable set $K\subset \R^3$  such that, up to a subsequence (not relabelled), for every open and bounded set $D\subset \R^3$ we have
$$
(\partial^*\Omega_k\cup\Gamma_k)\cap D \to K\cap D \qquad\text{in the sense of $\sigma^2$-convergence}.
$$
Since $J_{1_{\Omega_k \cap D}} \tsub (\partial^*\Omega_k\cup\Gamma_k)\cap D$ and $\nabla 1_{\Omega_k \cap D}=\nabla 1_{\Om\cap D}=0$, in view of \eqref{eq:s-omega_n} and of property (a) in Definition \ref{def:sigma2} of $\sigma^2$-convergence, we deduce that 
$$
\partial^*\Omega \cap D\tsub K\cap D.
$$
Since $D$ is arbitrary, we infer $\partial^*\Om\tsub K$. Now, let us decompose $K$ as 
$$
K=(K\cap\Omega^0)\cup\partial^*\Omega\cup(K\cap\Omega^1)
$$
and set
$$
\Gamma:=K\cap\Omega^1,
$$
so that $(\Om,\Gamma)\in \As(\R^3)$. 
\par
We divide the proof into two steps.

\vskip10pt\noindent{\bf Step 1.}
We claim that for every open and bounded set $D\subset \R^3$, there exist functions $v,v_k\in SBV(D)$ satisfying $\|v\|_\infty,\|v_k\|_\infty\le C$, $v=v_k=0$ a.e. outside $\Om$ and $\Om_k$ respectively, $\nabla v,\nabla v_k\in L^2(D;\R^3)$, $J_v \teq (\partial^*\Omega\cup\Gamma)\cap D$, $J_{v_k}\tsub (\partial^*\Omega_k\cup\Gamma_k)\cap D$, and such that
$$
v_k \to v\quad\text{strongly in $L^2(D)$},
$$
and
$$
\nabla v_k \weak \nabla v\quad\text{weakly in $L^2(D;\R^3)$}.
$$
According to the definition of $\sigma^2$-convergence, there exist $w,w_k\in SBV(D)$ satisfying $\|w\|_\infty,\|w_k\|_{\infty}\le C$, $J_w \teq K\cap D$, $J_{w_k}\tsub (\partial^*\Omega_k\cup\Gamma_k)\cap D$, and such that
$$
w_k\to  w \qquad\text{strongly in }L^1(D)
$$
and
$$
\nabla w_k\weak  \nabla w \qquad\text{weakly in }L^2(D;\R^3).
$$
The first convergence also holds strongly in $L^2(D)$ thanks to the uniform bound on $L^\infty$-norms. For $\varepsilon>0$ let us set
$$
v:=(w+\eps)1_{\Om\cap D}\qquad\text{and}\qquad v_k:=(w_k+\eps)1_{\Omega_k\cap D}.
$$
Clearly, we have
$$
\Gamma \cap D\tsub J_{w1_{\Omega\cap D}} \tsub (\partial^*\Omega\cup\Gamma)\cap D
$$
and so, for a.e. $\varepsilon>0$, we deduce (see Remark \ref{rem:jumpsum}) that
$$
J_v \teq J_{w1_{\Omega\cap D}}\cup J_{1_{\Omega\cap D}} \teq (\partial^*\Omega\cup\Gamma)\cap D.
$$
Since
$$
J_{v_k} \tsub (\partial^*\Omega_k\cup\Gamma_k)\cap D,
$$
$v_k\to v$ strongly in $L^2(D)$ and $\nabla v_k\weak \nabla v$ weakly in $L^2(D;\R^3)$, the claim follows by choosing $\eps$ outside a negligible set.

\vskip10pt\noindent{\bf Step 2.} We are now in a position to check the main property of the statement. 
\par
By point (b) in Proposition \ref{prop:hs}, the sequence $(\n_k)_{k\in\N}$ is bounded in $BV(\R^3;\R^3)$. Therefore, there exist $\n \in BV(\R^3;\R^3)$ such that $\n_k\weakst \n$ weakly* in $BV(\R^3;\R^3)$.
From \eqref{eq:s-omega_n}, we derive that
$$
\n_k\to \n\qquad\text{strongly in }L^1(\R^3;\R^3),
$$
implying, in particular, that
\begin{equation*}
\n=0 \quad \text{a.e. in $\Omega^c$}.
\qquad\text{and}\qquad
|\n|=1 \quad \text{a.e. in $\Omega$}.
\end{equation*}
Additionally, by locally applying Ambrosio's Theorem (since $\hdue(J_{\n_k})\le \hs(\partial^*\Om_k\cup \Gamma_k)\le C$), we infer $\n\in SBV(\R^3;\R^3)$ with
\begin{equation*}
\nabla \n_k\weak \nabla \n\qquad\text{weakly in }L^2(\R^3;\Mtre).
\end{equation*}
By the property of $\sigma^2$-convergence, it follows that
\begin{equation*}
J_{\n}\tsub\partial^*\Omega\cup\Gamma,
\end{equation*}
thus $\n\in \ns(\Om,\Gamma)$.
\par
Let us verify the lower semicontinuity inequality \eqref{eq:superficie}. For every $\varepsilon>0$, let us set
$$
\bm w_k:=\n_k+\varepsilon v_k \e_1\qquad\text{and}\qquad \bm w:=\n+\varepsilon v \e_1,
$$ 
where $v_k,v$ are the functions obtained in Step 1, and $\e_1$ is the first vector of the canonical basis of $\R^3$. For a.e. $\varepsilon>0$ we have
(see Remark \ref{rem:jumpsum})
$$
J_{\bm w}=J_{\n+\varepsilon v \e_1}\teq J_{\n}\cup J_v\teq (\partial^*\Omega\cup\Gamma)\cap D
\qquad\text{and}\qquad
J_{\bm w_k}=J_{\n_k+\varepsilon v_k \e_1}\tsub(\partial^*\Omega_k\cup\Gamma_k)\cap D.
$$
By the lower semicontinuity of surface energies in $SBV$ (see \cite[Theorem 5.22 and Example 5.23]{AFP}) we obtain
\begin{multline*}
\int_{(\partial^*\Omega\cup\Gamma)\cap D}[\varphi(\n^++\varepsilon v^+ \e_1,\normal)+\varphi(\n^-+\varepsilon v^- \e_1,\normal)]\,d\hdue\\
=
\int_{(\partial^*\Omega\cup\Gamma)\cap D}[\varphi(\bm w^+,\normal)+\varphi(\bm w^-,\normal)]\,d\hdue\\
\le
\liminf_{k\to+\infty}
\int_{(\partial^*\Omega_k\cup\Gamma_k)\cap D}[\varphi(\bm w^+_k,\normal_k)+\varphi(\bm w^-_k,\normal_k)]\,d\hdue\\
=
\liminf_{k\to+\infty}
\int_{(\partial^*\Omega_k\cup\Gamma_k)\cap D}[\varphi(\n^+_k+\varepsilon v^+_k \e_1,\normal_k)+\varphi(\n^-_k+\varepsilon v^-_k \e_1,\normal_k)]\,d\hdue.
\end{multline*}
The arbitrariness of $D$ ensures that as $\varepsilon\to0^+$, inequality \eqref{eq:superficie} in the theorem's statement holds thanks to the uniform boundedness of $v,v_k$ in $L^\infty(D)$, the continuity of $\varphi$ and the bound on the Hausdorff dimension of the boundary \eqref{eq:boundProb}. This concludes the proof.
\end{proof}

\begin{remark}
\label{rem:jumpsum}
{\rm
In the previous proof, we relied on the following result: if $D\subseteq \R^d$ is open and $u,v\in SBV(D)$ with $\hd(J_u),\hd(J_v)<+\infty$, then for a.e. $\eps>0$ we have
$J_{u+\eps v}\teq J_u\cup J_v$. We refer to \cite[Remark 3.13]{BucGiac-lambdak} for a proof of this result.
}
\end{remark}

\begin{remark}
\label{rem:traces}
{\rm
Let $(\Om,\Gamma)\in\As(\R^3)$ with $\Om$ open and bounded. Assume that $\n\in \ns(\Om,\Gamma)$ realizes $\tilde\free(\Om,\Gamma)$. If $\Om_k\subset\subset \Om$ is open and invades $\Om$ (i.e., if $K\subset \Om$ is compact, then $K\subset \Om_k$ for $k$ sufficiently large) with $\hdue(\partial^*\Om_k)\to \hdue(\partial^*\Om)$, then
\begin{equation}
\label{eq:conv-traces}
\lim_{k\to+\infty}\int_{\partial^* \Om_k}(\n\cdot \normal_k)^2\,d\hdue=\int_{\partial^* \Om}(\n\cdot \normal)^2\,d\hdue.
\end{equation}
To see this, note that $(\Om,J_{\n}\cap \Om)$ is the limit configuration of the sequence $(\Om_k,\Gamma_k)$ where $\Gamma_k:=J_{\n}\cap \Om_k$. Consequently, the sequence $(\n_k)_{k\in\N}$ with $\n_k:=\n 1_{\Om_k}$ satisfies $\n_k\in \ns(\Om_k,\Gamma_k)$ with $\n_k\to \n$ strongly in $L^1(\R^3;\R^3)$.
By Lemma \ref{lem:glambda}, the lower semicontinuity of surface energies in \eqref{eq:superficie} implies that for every $\alpha\in [-1/2,1]$
\begin{multline*}
\int_{\partial^*\Om}[1+\alpha (\n\cdot\normal)^2]\,d\hdue+\int_{J_{\n}\cap \Om^1}[2+\alpha (\n^+\cdot\normal)^2+\alpha (\n^-\cdot\normal)^2]\,d\hdue\\
\le 
\liminf_{k\to+\infty}\left[ \int_{\partial^*\Om_k}[1+\alpha (\n\cdot\normal_k)^2]\,d\hdue
+\int_{\Gamma_k}[2+\alpha (\n^+\cdot\normal_k)^2+\alpha (\n^-\cdot\normal_k)^2]\,d\hdue
\right].
\end{multline*}
Since $\Gamma_k\subseteq J_{\n}$, we get
$$
\int_{\partial^*\Om}[1+\alpha (\n\cdot \normal)^2]\,d\hdue\le \liminf_{k\to+\infty}\int_{\partial^*\Om_k}[1+\alpha (\n\cdot \normal_k)^2]\,d\hdue.
$$
The convergence of the perimeters, along with the fact that $\alpha$ can take both positive and negative values, entails the convergence stated in \eqref{eq:conv-traces}.
}
\end{remark}

\subsection{Approximation through regular configurations}
\label{sec:density}
In this section, we show how to approximate in energy a configuration represented by a director field $\n\in \Ns(\Om,\Gamma)$, where $(\Om,\Gamma)\in \As(\R^3)$. Our approach parallels the methodology presented
in \cite{braides:relaxation}, albeit with key distinctions: our functions take values in $\mathbb{S}^{2}$, and the surface energy depends on the traces of $\n$ on both sides of the inner boundary.
\par
We start by adapting the approximation result for SBV functions from \cite[Lemma 5.2]{BrCh} to our specific context.

\begin{lemma}
\label{lem:BC}
Let $\Om\subset\R^3$ be an open bounded set with a Lipschitz boundary. Suppose $\uu\in SBV(\Om;\R^3)$ satisfies $|\uu|=1$ a.e. in $\Omega$, $\nabla \uu\in L^2(\Om;\Mtre)$ and $\hdue\left(J_{\uu}\right)<\infty$. Then there exist a sequence $\left(R_n\right)_{n\in\mathbb{\N}}$ of rectifiable sets and a sequence $(\uu_n)_{n\in\mathbb{N}}$ of functions in $SBV(\Om;\R^3)$ such that:
\begin{itemize}
\item[(a)] every $R_n$ is relatively closed in $\Omega$, and $\hdue(R_n\setminus J_{\uu_n})\to 0$. Every $\uu_n\in C^1\left(\Omega\setminus R_n;\mathbb{R}^3\right)$ with $|\uu_n|=1$ a.e. on $\Omega\setminus R_n$. Moreover, the $\hdue$-measure of every compact subset $K\subseteq R_n$ coincides with its $2$-dimensional Minkowski content. 
\item[(b)] The following convergence properties hold: $\uu_n\to\uu$ strongly in $L^2\left(\Omega;\mathbb{R}^3\right)$, $\nabla \uu_n \to \nabla \uu$ strongly in $L^2(\Om;\Mtre)$, $\hdue(J_{\uu_n}\Delta J_{\uu})\to 0$, 
$$
\int_{J_{\uu}\cup J_{\uu_n}}\left[|\uu_n^+-\uu^+|+|\uu_n^--\uu^-|\right]\dd\hdue\to 0,
$$
and
\begin{equation}
\label{eq:traces_bradies}
\int_{\partial\Omega}|\uu_n-\uu|\dd\hdue\to 0.
\end{equation}
\end{itemize}
\end{lemma}

\begin{proof}
For $n\geq1$ let us consider a compact set $K_n\subseteq J_{\uu}$ such that
\begin{equation}
\label{eq:K_n}
\hdue\left(J_{\uu}\setminus K_n\right)\leq\frac{1}{n}.
\end{equation}
Let us define the space
$$
\hs:=\{\uu\in SBV(\Om;\R^3)\,:\, \nabla \uu \in L^2(\Om;\Mtre), |\uu|=1 \text{ a.e. in }\Om\}.
$$
Following \cite[Lemma 5.2]{BrCh}, we consider the minimum problem
\begin{multline}
\label{eq:BrCh}
\min_{\vv\in \hs} \left\{ \int_\Omega|\nabla\vv|^2\,dx+\hdue\left(J_{\vv}\setminus K_n\right)+n\int_\Omega|\vv-\uu|^2\,dx\right.\\\left.+\int_{K_n}\left[|\vv^+-\uu^+|+|\vv^--\uu^-|\right]\wedge 1\,d\hdue+\int_{\partial\Omega}|\vv-\uu|\wedge 1\,d\hdue\right\}.
\end{multline}
The existence of a solution $\uu_n$ to problem \eqref{eq:BrCh} follows immediately by employing the direct method and Ambrosio's theorem (the lower semicontinuity of the functional, for which truncations at level $1$ in the surface integrals are essential, is ensured by considering its one-dimensional sections as in \cite{BrCh}).
\par
Comparing $\uu_n$ and $\uu$, we get the estimate
\begin{multline}
\label{eq:un-min}
\int_\Omega|\nabla\uu_n|^2\,dx+\hdue(J_{\uu_n})+n\int_\Omega|\uu_n-\uu|^2\,dx\\
+\int_{K_n}\left[|\uu_n^+-\uu^+|+|\uu_n^--\uu^-|\right]\wedge 1\,d\hdue+\int_{\partial\Omega}|\uu_n-\uu|\wedge 1
\,d\hdue\\
\le \int_\Omega|\nabla\uu_n|^2\,dx+\hdue(J_{\uu_n}\setminus K_n)+\hdue(K_n)+n\int_\Omega|\uu_n-\uu|^2\,dx\\
+\int_{K_n}\left[|\uu_n^+-\uu^+|+|\uu_n^--\uu^-|\right]\wedge 1\,d\hdue+\int_{\partial\Omega}|\uu_n-\uu|\wedge 1
\,d\hdue\\
\le \int_\Omega|\nabla\uu|^2\,dx+\hdue\left(J_{\uu}\setminus K_n\right)+\hdue(K_n)= \int_\Omega|\nabla\uu|^2\,dx+\hdue(J_{\uu}).
\end{multline}
Consequently, we deduce 
\begin{equation}
\label{eq:s-un}
\uu_n\to \uu \qquad\text{strongly in }L^2(\Om;\R^3),
\end{equation}
and, in view of Ambrosio's theorem,
\begin{equation}
\label{eq:s-nablaun}
\nabla \uu_n\to \nabla \uu \qquad\text{strongly in }L^2(\Om;\Mtre),
\end{equation}
\begin{equation}
\label{eq:Jun}
\hdue(J_{\uu_n})\to \hdue(J_u),
\end{equation}
and
\begin{equation}
\label{eq:tr}
\int_{K_n}\left[|\uu_n^+-\uu^+|+|\uu_n^--\uu^-|\right]\wedge 1\,d\hdue+\int_{\partial\Omega}|\uu_n-\uu|\wedge 1
\,d\hdue\to 0.
\end{equation}
In particular, from \eqref{eq:un-min} and \eqref{eq:Jun}, we get (since $\hdue(K_n)\to \hdue(J_{\uu})$)
\begin{equation}
\label{eq:Jun2}
\hdue(J_{\uu_n}\setminus K_n)\to 0\qquad \text{and}\qquad \hdue(J_{\uu_n}\cap K_n)\to \hdue(J_{\uu}),
\end{equation}
which implies
\begin{multline}
\label{eq:JunDJu}
\hdue(J_{\uu_n}\Delta J_u)\to 0=\hdue(J_{\uu_n}\setminus J_{\uu})+\hdue(J_{\uu}\setminus J_{\uu_n})\\
\le \hdue(J_{\uu_n}\setminus K_n)+\hdue(J_{\uu}\setminus (J_{\uu_n}\cap K_n))\\
=\hdue(J_{\uu_n}\setminus K_n)+\hdue(J_{\uu})-\hdue(J_{\uu_n}\cap K_n)\to 0.
\end{multline}
Since $\uu_n$ and $\uu$ are bounded, it readily follows from \eqref{eq:tr} that
\begin{equation}
\label{eq:tr2}
\int_{\partial\Omega}|\uu_n-\uu|\,d\hdue\to 0.
\end{equation}
Considering \eqref{eq:tr} and \eqref{eq:Jun2}, we can further write
\begin{multline*}
\int_{J_{\uu}\cup J_{\uu_n}}\left[|\uu_n^+-\uu^+|+|\uu_n^--\uu^-|\right]\wedge 1d\hdue\\
\le
\int_{K_n}\left[|\uu_n^+-\uu^+|+|\uu_n^--\uu^-|\right]\wedge 1 d\hdue+\hdue((J_{\uu}\cup J_{\uu_n})\setminus K_n)\to 0.
\end{multline*}
In particular, up to a subsequence, we have
$$
|\uu_n^+-\uu^+|+|\uu_n^--\uu^-|\to 0 \qquad\text{$\hdue$-a.e. on } J_{\uu},
$$
and so, since $\uu_n$ and $\uu$ are bounded and $\hdue(J_{\uu_n}\setminus J_{\uu})\to 0$, we deduce that
$$
\int_{J_{\uu}\cup J_{\uu_n}}\left[|\uu_n^+-\uu^+|+|\uu_n^--\uu^-|\right]\,d\hdue\to 0.
$$
Point (b) follows by considering the relation above along with \eqref{eq:s-un}, \eqref{eq:s-nablaun}, \eqref{eq:JunDJu} and \eqref{eq:tr2}.
\par
Let us address point (a). Note that $\uu_n$ is a {\it local quasi minimizer} for the Mumford-Shah functional on $\Om\setminus K_n$
for $\mathbb{S}^{2}$-valued SBV functions: this is because the contributions of the integrals involving $K_n$ and $\partial\Omega$ in \eqref{eq:BrCh} are unaffected
by perturbations of $\uu_n$ compactly supported in $\Om\setminus K_n$. By extending the De Giorgi-Carriero-Leaci regularity result for $\mathbb{S}^{2}$-valued maps obtained in \cite[Lemma 4.5]{CaLe}, we deduce that
$J_{\uu_n}$ is essentially closed in $\Omega\setminus K_n$, meaning
\begin{equation}
\label{eq:ess_closed}
\hdue\left(\left(\overline{J_{\uu_n}}\setminus J_{\uu_n}\right)\cap\left(\Omega\setminus K_n\right)\right)=0,
\end{equation}
and that
$$
\uu_n\in C^1\left(\Omega\setminus(\overline{J_{\uu_n}}\cup K_n);\mathbb{R}^3\right).
$$ 
Additionally, according to \cite[Lemma 4.9]{CaLe}, the set $\overline{J_{\uu_n}}\cap (\Om\setminus K_n)$ satisfies the following lower density estimate: there exists $c>0$ such that for every $x\in \overline{J_{\uu_n}}\cap (\Om\setminus K_n)$ and every ball $B_r(x)\subset \Om\setminus K_n$ we have
\begin{equation}
\label{eq:Alf}
\hdue(J_{\uu_n}\cap B_r(x))\ge cr^{2}.
\end{equation}
Let us the define the set 
$$
R_n:=(\overline{J_{\bm u_n}}\cup K_n)\cap \Om,
$$
which is rectifiable and relatively closed in $\Om$. Clearly
\begin{equation*}
\uu_n\in C^1(\Om\setminus R_n;\R^3)\qquad\text{with}\qquad |\uu_n|=1\text{ on }\Om\setminus R_n.
\end{equation*}
From \eqref{eq:JunDJu} and \eqref{eq:ess_closed}, it follows that
$$
\hdue(R_n\setminus J_{\uu_n})=\hdue(K_n\setminus J_{\uu_n})\le \hdue(J_{\uu}\setminus J_{\uu_n})\to  0.
$$
To complete the proof of point (a), we need to verify the property concerning the Minkowski content of any compact subset $K$ of $R_n$. 
\par
It is not restrictive to assume that $K_n$ satisfies the following density lower bound: there exists $\bar r>0$ such that for every $x\in K_n$ and every $0<r<\bar r$
\begin{equation}
\label{eq:Knbelow}
\frac{\hdue(J_{\uu}\cap B_r(x))}{\pi r^{2}}\ge \frac{1}{2}.
\end{equation}
This assumption holds because $J_{\uu}$ is rectifiable and has finite $\hdue$-measure, properties which ensure that for $\hdue$-a.e. $x\in J_{\uu}$
$$
\lim_{r\to 0^+}\frac{\hdue(J_{\uu}\cap B_r(x))}{\pi  r^{2}}=1.
$$
By Severini-Egorov theorem, for every $n\ge 1$ there exists a compact set $K_n\subset J_{\uu}$ such that $\hdue(J_{\uu}\setminus K_n)<\frac{1}{n}$ and 
$$
\frac{\hdue(J_{\uu}\cap B_r(x))}{\pi r^{2}}\to 1\qquad\text{as $r\to 0^+$ uniformly on $x\in K_n$}.
$$
This implies \eqref{eq:Knbelow}.
\par
Let us consider the finite Radon measure $\nu:=\hdue\lfloor{(J_{\uu_n} \cup J_{\uu})}$. By \eqref{eq:Knbelow} and \eqref{eq:Alf}, there exists a constant $\tilde c>0$ such that for every compact set $K\subseteq R_n$ the following holds: for all $x\in K$ and $0<r<1$,
$$
\nu(B_r(x))\ge \tilde c r^{2}.
$$
The Minkowski content property stated in point (a) follows from the general result in \cite[Theorem 2.104]{AFP}.
\end{proof}

The following approximation result holds true.

\begin{theorem}
\label{thm:approx}
Let $D\subset \R^3$ be a bounded open set, and let $(\Om,\Gamma)\in \As(\R^3)$ with $\Om\subset\subset D$ and $\n\in \Ns(\Om,\Gamma)$. Then there exists a sequence $(A_k)_{k\in\N}$ of open sets with a smooth boundary such that $A_k\subset D$, 
\begin{equation}
\label{eq:Akprop}
|\Om\Delta A_k|\to 0,\qquad \limsup_{k\to+\infty}\hdue(\partial A_k)\le \hdue(\partial^*\Om)+2\hdue(\Gamma),
\end{equation}
and there exists a sequence $(\n_k)_{k\in\N}$ with $\n_k\in W^{1,2}(A_k;\mathbb S^2)$ such that
\begin{equation}
\label{eq:nk1}
\n_k 1_{A_k} \to \n\qquad\text{strongly in }L^2(\R^3;\R^3),
\end{equation}
\begin{equation}
\label{eq:nk2}
\nabla \n_k 1_{A_k} \to \nabla \n\qquad\text{strongly in }L^2(\R^3;M^{3\times 3}),
\end{equation}
and
\begin{equation}
\label{eq:nk3}
\limsup_{k\to+\infty} \int_{\partial A_k}(\n_k\cdot \normal_k)^2\,d\hdue\le 
\int_{\partial^*\Om}(\n\cdot \normal)^2\,d\hdue+\int_{\Gamma}[(\n^+\cdot \normal)^2+(\n^-\cdot\normal)^2]\,d\hdue.
\end{equation}
\end{theorem}

\begin{proof}
Let $dist(\Om,\partial D)>c>0$ (for example, identify $\Om$ with $\Om^1$). Let us divide the approximation process into several steps.

\vskip10pt\noindent{\bf Step 1.}  
We claim that $\Om$ can be assumed to be open with a smooth boundary. Without loss of generality, we assume that $\Gamma=J_{\n}\cap \Om^1$, where $\n\in \ns(\Om,\Gamma)$ is a director that realizes the minimum of the total energy in \eqref{eq:Etot} for the given $(\Omega,\Gamma)$.
Let us approximate $\Om$ ``from inside'' according to Proposition \ref{pro:comitorres}
with $\mu:=\hdue\lfloor (\partial^*\Om \cup \Gamma)$. We find a sequence of  smooth open sets $(\Om_k)_{k\in\N}\subset\R^3$ such that $dist(\Om_k,\partial D)>c$ and the following convergence properties hold:
$$
|\Om_k \Delta\Omega|\to 0,\quad \hdue(\partial \Om_k)\to \hdue(\partial^*\Omega),
$$
and
$$
\hdue(\partial^*\Om \cap \Om_k)\to 0, \qquad \hdue(\Gamma\setminus \Om_k)\to 0, \qquad \hdue(\partial \Om_k\setminus \Omega^1)\to 0.
$$
Let us define 
$$
\Gamma_k:= (\Gamma\cap \Om_k) \cup (\partial^*\Om \cap \Om_k).
$$
Consequently, $(\Om_k,\Gamma_k)\in \As(\R^3)$, and by construction,
\begin{equation}
\label{eq:geom1}
|\Om_k\Delta \Om|\to 0,\qquad\text{and} \qquad \lim_{k\to+\infty} \left[\hdue(\partial \Om_k)+2\hdue(\Gamma_k)\right]=\hdue(\partial^*\Om)+2\hdue(\Gamma).
\end{equation}
\par
Notice that $(\Om,\Gamma)$ is the limit configuration of $(\Om_k,\Gamma_k)_{k\in\N}$ according to Theorem \ref{thm:lsc}. Indeed, let $(\Om,K)\in\As(\R^3)$ be a limit configuration (up to subsequences) of $(\Om_k,\Gamma_k)_{k\in\N}$. Define the sequence $(\n_k)_{k\in\N}$ such that
$$
\n_k:=
\begin{cases}
\n &\text{ in }\Om_k\cap \Om\\
\bm e_1&\text{in }\Om_k \setminus \Om\\
0 &\text{in }\R^3\setminus \Om_k
\end{cases}
$$
where $\n\in \ns(\Om,\Gamma)$ is the director considered above that realizes the minimal energy of $(\Om,\Gamma)$, and $\bm e_1$ is the first element of the canonical base of $\R^3$. This construction ensures that $\n_k\in \ns(\Om_k,\Gamma_k)$ and that
\begin{equation}
\label{eq:Sn1}
\n_k\to \n\qquad\text{strongly in }L^2(\R^3;\R^3),
\end{equation}
and
\begin{equation}
\label{eq:Sn2}
\nabla\n_k\to \nabla\n\qquad\text{strongly in }L^2(\R^3;\Mtre).
\end{equation}
Thanks to Theorem \ref{thm:lsc} we deduce that $\n\in \ns(\Om,K)$, implying $\Gamma=J_{\n}\cap \Om^1 \tsub K$. On the other hand, by choosing $\varphi(\bm v,\bm p):=|\bm v||\bm p|$ in \eqref{eq:superficie}, we get
$$
\hdue(\partial^*\Om)+2\hdue(K)\le \liminf_{k\to+\infty}\left[ \hdue(\partial\Om_k)+2\hdue(\Gamma_k)\right]=\hdue(\partial^*\Om)+2\hdue(\Gamma),
$$
which yields $\hdue(K)\le \hdue(\Gamma)$. As a consequence, we deduce that $K\teq \Gamma$, and the result follows.
\par
Thanks to Lemma \ref{lem:glambda} and the lower semicontinuity result from Theorem \ref{thm:lsc}, we obtain for every $\alpha\in [-1/2,1]$ that
\begin{multline*}
\int_{\partial^*\Om}[1+\alpha (\n\cdot\normal)^2]\,d\hdue+
\int_{\Gamma}[2+\alpha (\n^+\cdot\normal)^2+\alpha (\n^-\cdot\normal)^2] \,d\hdue\\
\le \liminf_{k\to+\infty}
\left[
\int_{\partial\Om_k}[1+\alpha (\n_k\cdot\normal_k)^2]\,d\hdue+
\int_{\Gamma_k}
[2+\alpha (\n_k^+\cdot\normal_k)^2+\alpha (\n_k^-\cdot\normal_k)^2] \,d\hdue
\right].
\end{multline*}
Given that $\alpha$ can be both positive and negative, and taking into account \eqref{eq:geom1}, this inequality leads to
\begin{multline}
\label{eq:surf1}
\int_{\partial^*\Om}(\n\cdot\normal)^2\,d\hdue+
\int_{\Gamma}[(\n^+\cdot\normal)^2+(\n^-\cdot\normal)^2]\,d\hdue\\
=\lim_{k\to+\infty}
\left[
\int_{\partial\Om_k}(\n_k\cdot\normal)^2\,d\hdue+
\int_{\Gamma_k}[(\n^+_k\cdot\normal_k)^2+(\n^-_k\cdot\normal_k)^2]\,d\hdue
\right].
\end{multline}
The claim follows in view of \eqref{eq:geom1}, \eqref{eq:Sn1}, \eqref{eq:Sn2} and \eqref{eq:surf1}.

\vskip10pt\noindent{\bf Step 2.} 
Let $\e>0$. We claim that we can assume $\Om\subset\subset D$ is smooth with $\Gamma=R\cap \Om$, where $R\subset \Omb$ is a closed rectifiable set such that $\hdue(R\setminus J_{\n})<\eps$ and $\hdue(R\cap \partial\Om)=0$. Moreover, we can assume that the two dimensional Minkowski content of $R$ coincides with $\hdue(R)$.
\par
Indeed, by Step 1 we know that $\Om$ can be assumed to be smooth. Let us consider the sequences $(R_k)_{k\in\N}$ of rectifiable sets in $\Omega$ and $(\n_k)_{k\in\N}$ of functions $\in SBV(\Om;\R^3)$ that approximate $J_{\n}$ and $\n$, respectively, according to Lemma \ref{lem:BC}. Since $\hdue(R_k\Delta J_{\n})\to 0$, we have
\begin{equation*}
\lim_{k\to+\infty}\hdue(R_k)=\hdue(J_{\n})\le \hdue(\Gamma).
\end{equation*}
Moreover, $\n_k\in \Ns(\Om,R_k)$ with
$$
\n_k\to \n\qquad\text{strongly in }L^2(\R^3;\R^3)
\qquad\text{and}\qquad \nabla \n_k\to \nabla \n\qquad\text{strongly in }L^2(\R^3;\Mtre).
$$
In view of the convergence of the traces and the relations $\|\n_k\|_\infty\le 1$ and $\hdue(R_k\Delta J_{\n})\to 0$, we obtain
\begin{multline*}
\limsup_{k\to+\infty} \left[\int_{\partial \Om}(\n_k\cdot\normal)^2\,d\hdue+\int_{R_k}[(\n_k^+\cdot\normal_k)^2+(\n_k^-\cdot\normal_k)^2]\,d\hdue\right]\\
=
\int_{\partial \Om}(\n\cdot\normal)^2\,d\hdue+\int_{J_{\n}}[(\n^+\cdot\normal)^2+(\n^-\cdot\normal)^2]\,d\hdue\\
\le
\int_{\partial \Om}(\n\cdot\normal)^2\,d\hdue+\int_{\Gamma}[(\n^+\cdot\normal)^2+(\n^-\cdot\normal)^2]\,d\hdue.
\end{multline*}
Since $\Om$ has a smooth boundary, the sets defined as
$$
\Om_\eta:=\{x\in \Om\,:\, dist(x,\partial\Om)>\eta\}
$$
also have smooth boundaries for $\eta$ small enough (see e.g. \cite{KP}) with $\hdue(\partial\Om_\eta)\to \hdue(\partial\Om)$ as $\eta\to 0^+$. Moreover, 
$\hdue(\partial \Om_{\eta}\cap R_k)=0$ for a.e. $\eta$. By choosing $\eta_k\to 0$ such that $\hdue(\partial \Om_{\eta_k}\cap R_k)=0$, and using Remark  \ref{rem:traces} to ensure the convergence of the traces, the claim follows by considering the configuration $(\Om_{\eta_k},R_k\cap \Omb_{\eta_k})$ and the restriction of $\n_k$ to $\Om_{\eta_k}$ for sufficiently large $k$, since the Minkowski content of $R_k\cap \overline{\Om}_{\eta_k}$ coincides with its $\hdue$-measure thanks to the general properties of $R_k$ stated in Lemma \ref{lem:BC}.

\vskip10pt\noindent{\bf Step 3.} Assume $(\Om,\Gamma)\in \As(\R^3)$ is of the form given by the approximation result of Step 2. Specifically let $\Om\subset\subset D$ be smooth and $\Gamma=R\cap \Om$, where $R\subset \Omb$ is a suitable closed rectifiable set, such that $\hdue(\Gamma\setminus J_{\n})<\eps$ where $\eps>0$.
 \par
We claim that we can find a sequence $(A_k)_{k\in\N}$ of sets $A_k\subset\subset \Om$ with smooth boundaries such that 
 \begin{equation*}
|\Om\setminus A_k|\to 0,\qquad \limsup_{k\to+\infty}\hdue(\partial A_k)\le \hdue(\partial \Om)+2\hd(\Gamma)
\end{equation*}
and a sequence $(n_k)_{k\in\N}$ in $W^{1,2}(A_k;\mathbb S^2)$ such that
 \begin{equation*}
\n_k 1_{A_k} \to \n\qquad\text{strongly in }L^2(\R^3;\R^3),
\end{equation*}
\begin{equation*}
\nabla \n_k 1_{A_k} \to \nabla \n\qquad\text{strongly in }L^2(\R^3;M^{3\times 3}),
\end{equation*}
and
\begin{equation*}
\limsup_{k\to+\infty} \int_{\partial A_k}(\n_k\cdot \normal_k)^2\,d\hdue\le 
\int_{\partial\Om}(\n\cdot \normal)^2\,d\hdue+\int_{\Gamma}[(\n^+\cdot \normal)^2+(\n^-\cdot\normal)^2]\,d\hdue+3\sqrt\eps.
\end{equation*}
Indeed, setting $d_R(x):=dist(x,R)$, we can apply the coarea formula to write
\begin{equation*}
\int_0^{1/k} \hdue(\{d_R=s\})\,ds=\int_{\{d_R<\frac{1}{k}\}}|\nabla d_R|\,dx=\left|\left\{d_R<\frac{1}{k}\right\}\right|,
\end{equation*}
implying that we can select $s_k\le \frac{1}{k}$ such that
$$
 \hdue(\{d_R=s_k\})\le \frac{|\{d_R<1/k\}|}{1/k}.
$$
Since the two dimensional Minkowski content of $R$ coincides with $\hdue(R)$, it follows that
$$
\limsup_{k\to+\infty}\hdue(\{d_R=s_k\})\le 2\hdue(R)=2\hdue(\Gamma).
$$
Define 
$$
\Om_k:=\Om\setminus \{x\in\R^3\,:\, d_R<s_k\};
$$
then $\Om_k$ is an open set with finite perimeter such that $|\Om\setminus \Om_k|\to 0$, and (recall $\hdue(R\cap \partial \Om)=0$)
\begin{equation}
\label{eq:ak2}
\limsup_{k\to +\infty} \hdue(\partial^* \Om_k)\le \limsup_{k\to +\infty} \hdue(\partial \Om_k)\le \hdue(\partial \Om)+2\hdue(\Gamma).
\end{equation}
Let $(\Om,K)$ be a limit configuration of $(\Om_k,\emptyset)$ (up to a subsequence not relabelled) according to Theorem \ref{thm:lsc}. Since $J_{\bm n}\cap\Om\tsub K\tsub \Gamma$, it follows that
$$
\hdue(\Gamma\setminus K)<\eps.
$$
By considering $\n_k:=\n 1_{\Om_k}$, and applying Lemma \ref{lem:glambda} and the lower semicontinuity of Theorem \ref{thm:lsc}, for every $\alpha\in [-1/2,1]$ we obtain
\begin{multline*}
\int_{\partial\Om}[1+\alpha (\n \cdot \normal)^2]\,d\hdue+
\int_{\Gamma}[2+\alpha (\n^+ \cdot \normal)^2+\alpha (\n^- \cdot \normal)^2)]\,d\hdue-3\eps\\
\le
\int_{\partial\Om}[1+\alpha (\n \cdot \normal)^2]\,d\hdue+
\int_{K}[2+\alpha (\n^+ \cdot \normal)^2+\alpha (\n^- \cdot \normal)^2)]\,d\hdue\\
\le \liminf_{k\to+\infty}
\int_{\partial^*\Om_k}[1+\alpha (\n_k\cdot \normal_k)^2]\,d\hdue.
\end{multline*}
By choosing $\alpha=-\sqrt\eps$ in the above inequality, and considering \eqref{eq:ak2}, we deduce
$$
 \limsup_{k\to+\infty}\int_{\partial^*\Om_k}(\n_k\cdot \normal_k)^2\,d\hdue
 \le  \int_{\partial\Om}(\n\cdot \normal)^2\,d\hdue+ \int_{\Gamma} [(\n^+\cdot \normal)^2+(\n^-\cdot \nu)]^2\,d\hdue+3\sqrt{\eps}.
$$
Note that $\Om_k$ is open, but its boundary may not be smooth in principle. Therefore, a further approximation of each $\Om_k$ ``from inside'' is required to obtain the result. To this end, we assert that we can assume the mild regularity
$$
\hdue(\partial^*\Om_k)=\hdue(\partial\Om_k).
$$
Indeed, this follows directly by construction, as for a.e. $s\in ]0,1/k[$ the coarea formula for $d_R$ in $\Om$ (valid both in classical and BV senses) yields
$$
\hdue(\partial^*\{d_R<s\}\cap \Om)=\hdue(\{d_R=s\}\cap \Om)=\hdue(\partial\{d_R<s\}\cap \Om).
$$
According to \cite[Theorem 1.1]{Sch}, we can approximate $\Om_k$ from the inside through smooth sets $\Om_k^h$ such that $\hdue(\partial\Om_k^h)\to \hdue(\partial^*\Om_k)$ for $h\to+\infty$. By Remark \ref{rem:traces}, it follows that
$$
\lim_{h\to+\infty} \int_{\partial \Om_k^h} (\n_k\cdot\normal_h)^2\,d\hdue=\int_{\partial^*\Om_k}(\n_k\cdot \normal_k)^2\,d\hdue.
$$
The claim of step 3 is then justified by setting $A_k:=\Om_k^{h_k}$ and considering the restriction of $\n_k$ to $A_k$, where $h_k$ is sufficiently large.

\vskip10pt\noindent{\bf Step 4: Conclusion.} 
The desired approximation follows from those constructed in the preceding steps by choosing $\eps=\eps_k\to 0$ and employing a diagonal argument. Indeed, this process yields a sequence $(A_k)_{k\in\N}$ of open sets in $D$ with smooth boundaries such that $dist(A_k,\partial D)>c$, satisfying the convergence properties in \eqref{eq:Akprop}, \eqref{eq:nk1}, \eqref{eq:nk2} and \eqref{eq:nk3}. 
\end{proof}

\section{Proof of Theorem \ref{thm:main2}}
\label{sec:mainpf}
The present section is devoted to proving the existence of stable configurations in the generalized class of admissible geometries.

\begin{proof}[Proof of Theorem \ref{thm:main2}]
Let $(\Omega_k,\Gamma_k)_{k\in\N}$ be a minimizing sequence for problem \eqref{eq:mainpb2} with
$$
\tilde\free(\Om_k,\Gamma_k)\to \beta>0
$$
(recall that $\tilde\free(\Om_k,\Gamma_k)\ge \hdue(\partial^*\Om_k) \ge c |\Om_k|^{2/3}=cm^{2/3}$ by the isoperimetric inequality).
Consequently, we have
\begin{equation}
\label{eq:boundper3}
\hdue(\partial^*\Om_k\cup \Gamma_k)\le C
\end{equation}
\par
Let us apply a concentration-compactness alternative argument to the sequence $(1_{\Omega_k})_{k\in\N}$ of sets of finite perimeter (see e.g.  \cite[Theorem 3.1]{de2014existence}). For every $r>0$, let us consider the monotone increasing functions $\alpha_k:[0,+\infty[\to [0,+\infty[$ defined by
\begin{equation}
\label{eq:alpha_kcube}
\alpha_k(r):=\sup_{y \in \R^d}|\Om_k \cap Q_r(y)|,
\end{equation}
where $Q_r(y)$ is the cube centered at $y$ with side length $r$. By Helly's theorem, up to a subsequence, we can assume that
\begin{equation*}
\alpha_k \to \alpha
\qquad\text{pointwise on }[0,+\infty[
\end{equation*}
for a suitable monotone increasing function $\alpha:[0,+\infty[\to [0,+\infty[$. 
\par
The following situations may occur:
\begin{itemize}
\item[(a)] {\it Vanishing}: $\lim_{r\to+\infty}\alpha(r)=0$;
\item[(b)] {\it Dichotomy}: $\lim_{r\to+\infty}\alpha(r)=\bar \alpha \in ]0,m[$;
\item[(c)] {\it Compactness}: $\lim_{r\to+\infty}\alpha(r)=m$.
\end{itemize}
Let us deal with the three cases separately.

\vskip10pt\noindent{\bf Step 1: Vanishing cannot occur.} Assume, for contraddiction, that the vanishing case does occur. This implies that for every $r>0$ we have
\begin{equation}
\label{eq:starvanish}
\sup_{y\in\R^3}\left|\Omega_k\cap Q_r(y)\right|\to 0.
\end{equation}
Consider a grid in $\R^3$ composed of cubes $Q$ such that $\hdue(\partial^*\Om_k\cap \partial Q)=0$ (this is always possible because $\Om_k$ has finite perimeter).
For every $k\in\N$, we can find a cube $Q_1(y_k)$ with $y_k\in\R^3$ such that
$$
\hdue(\partial^*\Om_k\cap Q_1(y_k))\le K |\Om_k\cap Q_1(y_k)|,
$$
where $K$ is a fixed constant such that $K>\frac{2}{m}\tilde\free(\Om_k,\Gamma_k)$. Indeed, if the inequality above holds with a strict $>$ for every cube in the grid, summing over all the cubes would lead to a contradiction. By applying the immersion of $BV$ into $L^{3/2}$ on $Q_1(y_k)$ to $1_{\Om_k}$, we get
$$
C_1 |\Om_k\cap Q_1(y_k)|^{\frac{2}{3}}\le \hdue(\partial^*\Om_k\cap Q_1(y_k))+ |\Om_k\cap Q_1(y_k)|\le (K+1)  |\Om_k\cap Q_1(y_k)|,
$$
which is in contraddiction with \eqref{eq:starvanish}.

\vskip10pt\noindent{\bf Step 2: Dichotomy cannot occur.}
Assume, for contradiction, that the dichotomy case does occur. Then there exists $\tilde\alpha\in]0,m[$ such that the following assertion holds true: we can find $x_k\in\R^3$ and $0<r_k<R_k$, $R_k-r_k\to+\infty$, such that setting
$$
\Omega_{k,1}:=\Omega_k\cap B_{r_k}(x_k),\qquad \Gamma_{k,1}:=\Gamma_k\cap B_{r_k}(x_k)
$$
and
$$
\Omega_{k,2}:=\Omega_k\setminus \bar B_{R_k}(x_k),\qquad \Gamma_{k,2}:=\Gamma_k\setminus \bar B_{R_k}(x_k),
$$
we have
$$
\left||\Omega_{k,1}|-\tilde\alpha\right|\to 0,\qquad \left||\Omega_{k,2}|-(m-\tilde\alpha)\right|\to 0,
$$
with
\begin{equation}
\label{eq:cut}
\hdue(\Omega_k\cap\partial B_{r_k}(x_k))\to 0,\qquad \hdue(\Omega_k\cap\partial B_{R_k}(x_k))\to 0.
\end{equation}
Notice that 
$$
\hdue(\partial^*\Om_k) \ge \hdue(\partial^*\Om_{k,1})+\hdue(\partial^*\Om_{k,2})+e_k
$$
where $e_k\to 0$. 
\par
Let $\eta_{k,i}>0$ be such that $|\eta_{k,i}\Om_{k,i}|=|\Om_k|=m$, and consider the dilated configurations $(\eta_{k,i}\Om_{k,i},\eta_{k,i}\Gamma_{k,i})\in \As_m(\R^3)$. A simple rescaling argument shows that
$$
\eta_{k,i}^2\tilde\free(\Om_{k,i},\Gamma_{k,i})\le \tilde\free(\eta_{k,i}\Om_{k,i},\eta_{k,i}\Gamma_{k,i})\le \eta_{k,i}\tilde\free(\Om_{k,i},\Gamma_{k,i}).
$$
Let $\n_k\in \Ns(\Om_k,\Gamma_k)$ realize the minimal energy $\tilde\free(\Om_k,\Gamma_k)$. Since the restrictions of $\n_k$ to $\Om_{k,1}$ and $\Om_{k,2}$ are admissible directors, we can write, thanks to \eqref{eq:cut},
\begin{multline*}
\tilde\free(\Om_k,\Gamma_k) \ge \tilde\free(\Om_{k,1},\Gamma_{k,1})+\tilde\free(\Om_{k,2},\Gamma_{k,2})+\tilde e_k\\
\ge \eta_{k,1}^{-1}\tilde\free(\eta_{k,1}\Om_{k,1},\eta_{k,1}\Gamma_{k,1})+\eta_{k,2}^{-1}\tilde\free(\eta_{k,2}\Om_{k,2},\eta_{k,2}\Gamma_{k,2})+\tilde e_k,
\end{multline*}
where $\tilde e_k\to 0$. Taking the limit as $k\to\infty$, we get
$$
\beta\ge \left( \frac{\tilde \alpha}{m}\right)^{\frac{1}{3}}\beta+\left( \frac{m-\tilde \alpha}{m}\right)^{\frac{1}{3}}\beta
$$
which is a contradiction.

\vskip10pt\noindent{\bf Step 3. Compactness and existence of a minimizer.} It follows from Steps $1$ and $2$ that compactness occurs. Therefore, there exists a set of finite perimeter $\Omega \subset\R^3$ such that, up to a translation,
\begin{equation}
\label{eq:Omega_kconvergence}
1_{\Omega_k}\to 1_\Omega \qquad\text{strongly in }L^1(\R^3).
\end{equation}
Let $\Gamma\subset \Om^1$ be given by Theorem \ref{thm:lsc}, so that $(\Om,\Gamma)\in \As_m(\R^3)$. According to Theorem \ref{th:th_opt_director}, let $\n_k\in \ns(\Om_k,\Gamma_k)$ be such that
$$
\tilde\free(\Om_k,\Gamma_k)=\tilde\free[(\Om_k,\Gamma_k),\n_k].
$$
Using the estimates on the bulk energy given by Remark \ref{rem:bulk} and the bound \eqref{eq:boundper3} we obtain $\|\nabla \n_k\|_2^2\le C$ for some $C>0$ independent of $k$.
\par
Let $\n\in \ns(\Om,\Gamma)$ be the limit function of $(\n_k)_{k\in\N}$ (up to subsequences) given by Theorem \ref{thm:lsc}.
Accordingly, by the lower semicontinuity of quasiconvex bulk energies in SBV (see \cite[Theorem 5.12]{AFP}) we have
\begin{multline*}
\int_{\Om}W(\n,\nabla \n)\,dx=\int_{D}W(\n,\nabla \n)\,dx\\
\le \liminf_{k\to+\infty}\int_{D}W(\n_k,\nabla \n_k)\,dx=\liminf_{k\to+\infty}\int_{\Om_k}W(\n_k,\nabla \n_k)\,dx,
\end{multline*}
while for $\lambda\in [-1/2,1]$ we obtain, by Lemma \ref{lem:glambda} and Theorem \ref{thm:lsc}, that
\begin{multline*}
\int_{\partial^* \Om}g(\n,\normal)\,d\hdue+\int_{\Gamma}[g(\n^+,\normal)+g(\n^-,\normal)]\,d\hdue
\\
\le \liminf_{k\to+\infty}\left(
\int_{\partial^* \Om_k}g(\n_k,\normal_k)\,d\hdue+\int_{\Gamma_k}[g(\n^+_k,\normal_k)+g(\n^-_k,\normal_k)]\,d\hdue
\right).
\end{multline*}
Thus 
$$
\tilde\free(\Om,\Gamma)\le \tilde\free[(\Om,\Gamma),\n]\le \liminf_{k\to+\infty}\tilde\free[(\Om_k,\Gamma_k),\n_k]=\liminf_{k\to+\infty}\tilde\free(\Om_k,\Gamma_k),
$$
implying that $(\Om,\Gamma)\in \As_m(D)$ is a minimizer for problem \eqref{eq:mainpb2}.

\vskip10pt\noindent{\bf Step 4: Boundedness of minimizers.} 
Let $(\Om,\Gamma)\in \As_m(\R^3)$ be a minimizer for problem \eqref{eq:mainpb2}, which is well posed according to Step 3. Let us prove that $\Om$ is bounded. Assume, for sake of contradiction, that $\Om$ is unbounded: without loss of generality, we can assume that $\Om$ is unbounded in the positive $x_1$ direction, i.e.,
$$
h(t)=|\Om \cap \{x_1>t\}|>0\qquad\text{for every $t\in\R$}.
$$
Clearly, $h(t)\to 0$ as $t\to +\infty$. For sufficiently large $t$, let
$$
\Om_t:=\Om \cap \{x_1<t\}\qquad\text{and}\qquad \Gamma_t:=\Gamma \cap  \{x_1<t\}.
$$
Define $\eta_t:=\left(\frac{m}{|\Om_t|}\right)^{1/3}$. Then $(\eta_t \Om_t, \eta_t\Gamma_t)\in \As_m(\R^3)$, and by the minimality of $(\Om,\Gamma)$ and a simple rescaling argument we obtain
\begin{equation}
\label{eq:ineq-tlarge}
\tilde\free(\Om,\Gamma)\le \tilde\free(\eta_t\Om_t,\eta_t\Gamma_t)\le \eta_t\tilde\free(\Om_t,\Gamma_t).
\end{equation}
Let $\n\in \Ns(\Om,\Gamma)$ realize the total energy of $(\Om,\Gamma)$ yielding $\tilde\free(\Om,\Gamma)$, and define
\begin{multline*}
E(t):=\int_{\Om\cap \{x_1<t\}}W(\n,\nabla \n)\,dx+\int_{\partial^*\Om \cap \{x_1<t\}}[1+\lambda (\n\cdot \normal)^2]\,d\hdue\\
+\int_{\Gamma \cap \{x_1<t\}}[2+\lambda (\n^+\cdot \normal)^2+\lambda  (\n^+\cdot \normal)^2]\,d\hdue.
\end{multline*}
From \eqref{eq:ineq-tlarge} and given that $\lambda\in [-1/2,1]$), for a.e. $t$ large enough we have
\begin{multline*}
E(t)+\frac{1}{2}\hdue(\partial^*\Om\cap \{x_1>t\})\le \left(\frac{m}{m-h(t)}\right)^{1/3}\left[ E(t)+C_1 \hdue(\Om \cap \{x_1=t\})\right]\\
\le(1+C_2 h(t))(E(t)+C_1\hdue(\Om \cap \{x_1=t\})).
\end{multline*}
This implies
$$
\frac{1}{2}\hdue(\partial^*\Om\cap \{x_1>t\})\le C_3\left( \hdue(\Om \cap \{x_1=t\})+h(t)\right)
$$
for suitable $C_1,C_2,C_3>0$. Adding $\hdue(\Om \cap \{x_1=t\})$ to both sides and applying the isoperimetric inequality to the set $\Om\cap \{x_1>t\}$, we get
$$
h(t)^{\frac{2}{3}}\le C (h(t) +\hdue(\Om \cap \{x_1=t\})),
$$
so that for sufficiently large $t$
$$
h(t)^{\frac{2}{3}}\le \tilde C\hdue(\Om \cap \{x_1=t\}),
$$
where $C,\tilde C>0$. Since $\hdue(\Om \cap \{x_1=t\})=-h'(t)$, we reach a contradiction:
$$
+\infty=\int_{t_0}^{+\infty}\,dt\le -\int_{t_0}^{+\infty} \frac{h'(t)}{h(t)^{2/3}}\,dt=-3\left[h(t)^{1/3}\right]_{t_0}^{+\infty}<+\infty.
$$

\vskip10pt\noindent{\bf Step 5: Conclusion.} To conclude the proof, we need to verify equality \eqref{eq:reg-eq2}. Let $(\Om,\Gamma)\in \As_m(\R^3)$ be a minimizer for problem \eqref{eq:mainpb2}. By Step 4, we know that $\Om$ is bounded. Therefore, by Theorem \ref{thm:approx}, for every $\eps>0$ there exists a smooth set $A_\eps\subset \R^3$ such that $m-\eps<|A_\eps|<m+\eps$ and
$$
\free(A_\eps)< \tilde\free(\Om,\Gamma)+\eps.
$$
Define $\eta_\eps:=\left( \frac{m}{|A_\eps|}\right)^{\frac{1}{3}}$, so that $\eta_\eps A\in \As_m^{reg}(\R^3)$. Since
$$
\eta_\eps^2\free(A_\eps)\le \free(\eta_\eps A_\eps)\le \eta_\eps\free(A_\eps),
$$
we infer that
$$
\inf_{\As^{reg}_m(\R^3)}\free\le \eta_\eps( \tilde\free(\Om,\Gamma)+\eps).
$$
Thus, letting $\eps\to 0^+$ we obtain the desired result.
\end{proof}

\begin{remark}
\label{}
{\rm
If $(\Om,\Gamma)$ is a minimum of the main problem \eqref{eq:mainpb2}, and if $\n\in \ns(\Om,\Gamma)$ is a director which provides the associated free energy, we get easily that $\Gamma=J_{\n}\cap \Om^1$. This shows that the compactness needed is Step 3 can be simplified by looking at the sequence $(\n_k)_{k\in \N}$ of associated directors, avoiding the technical tool of $\sigma^2$-convergence of rectifiable sets (on which Theorem \ref{thm:lsc} is based).
As mentioned in the Introduction, considering configurations $(\Om,\Gamma)$ with $\Gamma$ possibly larger than the jump set of the director providing the associated free energy permits to maintain a closer link with the classical approach, and could be important for dynamical problems of liquid crystal drops in addressing the evolving behaviour of inner boundaries
}
\end{remark}

As mentioned in Section \ref{sec:main}, the existence of stable configurations can be established (with a proof that avoids the concentration-compactness alternative) even when the admissible shapes are contained in a given box, which we consider as a {\it design region} for the problem.
\par
Let  $D\subset\R^3$ be a bounded open set, and let $0<m<|D|$. We set
\begin{equation}
\label{eq:AsmD}
\As_m(D):=\{(\Om,\Gamma)\in \As(\R^3), \Om\subset D, |\Om|=m\}
\end{equation}
and
$$
\As^{reg}_m(D):=\{A\subset D, |A|=m, \text{$A$ is open with a smooth boundary}\}.
$$

The following result holds true.

\begin{theorem}[\bf Stable configurations in a box]
\label{thm:main}
Let $D\subset\R^3$ be a convex open bounded set, and let $0<m<|D|$. 
Assume that Ericksen's inequalities \eqref{eq:eineq} hold and let the constants of the surface energy satisfy $\gamma>0$ and $-\frac{1}{2}\le \lambda\le 1$. Then the minimum problem
\begin{equation}
\label{eq:mainpb}
\min_{\As_m(D)} \tilde\free
\end{equation}
is well posed. Moreover,
\begin{equation}
\label{eq:reg-eq}
\min_{\As_m(D)} \tilde\free=\inf_{\As_{m}^{reg}(D)} \free.
\end{equation}
\end{theorem}

\begin{proof}
Let $(\Om_k,\Gamma_k)_{k\in\N}$ be a minimizing sequence for the problem. Given that $\lambda\ge -1/2$, we can easily obtain the bound
\begin{equation}
\label{eq:boundper2}
\hdue(\partial^*\Om_k\cup \Gamma_k)\le C
\end{equation}
for some constant $C$ independent of $k$. Hence, there exists a set $\Om\subset D$ of finite perimeter such that, up to a subsequence,
$$
1_{\Om_k}\to 1_\Om\qquad\text{strongly in }L^1(\R^3).
$$
We can then follows the procedure outlined in Step 4 of the proof of Theorem \ref{thm:main2} to obtain a rectifiable set $\Gamma$ such that $(\Om,\Gamma)$ is a minimizer. Consequently, problem \eqref{eq:mainpb} is well posed.
\par
Let us address equality \eqref{eq:reg-eq}. Without loss of generality, we assume that the convex box $D$ contains the origin. Let $(\Om,\Gamma)\in \As_m(D)$ be a minimizer. For every $\eps>0$ let $\eta_\eps<1$ be such that $(\eta_\eps\Om,\eta_\eps\Gamma)\in \As_{m-\eps}(D)$. A simple rescaling argument shows that
$$
\eta_\eps^2\tilde\free(\Om,\Gamma)\le \tilde\free(\eta_\eps\Om,\eta_\eps\Gamma)\le \eta_\eps\tilde\free(\Om,\Gamma).
$$
Furthermore, by Theorem \ref{thm:approx}, there exists a smooth set $A\subset D$ such that $m-\frac{3}{2}\eps<|A|<m-\frac{1}{2}\eps$ and
$$
\free(A)< \tilde\free(\eta_\eps\Om,\eta_\eps \Gamma)+\eps< \tilde\free(\Om,\Gamma)+\eps.
$$
Since $A$ does not satisfy the volume constraint, we seek to replace $A$ with a smooth domain of volume $m$ such that the total free energy of the new domain is only a small perturbation of the total free energy of $A$.
Let $A\subset\subset P\subset\subset D$, where $P$ is the union of a finite number of cubes, each of side length $l$. We then partition the cubes of $P$ into $n_\eps^3$ small cubes, each with side length $l/n_\eps$, where $n_\eps\to +\infty$ is chosen such that
$$
\frac{|D|-m}{n_\eps^3}>2\eps.
$$
There exists one such cube, say $Q_i$, for which
$$
|A^c\cap Q_i|\ge \frac{1}{n_\eps^3}|A^c|> \frac{|D|-m}{n_\eps^3}>2\eps.
$$
We can then select an open parallelepid $R_i\subset\subset Q_i$ such that $\hs^2(\partial R_i\cap \partial A)=0$ and $|A\cup R_i|=m+e_\eps$, where $e_\eps\to 0$ as $\eps\to 0$. The set $A\cup R_i\subset D$ has finite perimeter with
$$
\hs^2(\partial^*(A\cup R_i))\le \hs^2(\partial A)+\hs^2(\partial R_i)=\hs^2(\partial A)+ \frac{6 l^3}{n_\eps^2}.
$$
Let $\n_A\in W^{1,2}(A;{\mathbb S}^2)$ be the director of $A$ which realizes the minimal energy according to \eqref{eq:Etot}. By \cite[Lemma 1.1]{HKL} we can extend $\n_A$ to an element of $W^{1,2}(D;{\mathbb S}^2)$. Since $\n_A$ is also an admissible director also for $A\cup R_i$, we deduce that
$$
\tilde\free(A\cup R_i,\emptyset)< \free(A)+\hat e_\eps,
$$
where $\hat e_\eps\to 0$ as $\eps\to 0$.
\par
Since $\hs^2(\partial(A\cup R_i))=\hs^2(\partial^*(A\cup R_i))$ by construction, and using \cite[Theorem 1.1]{Sch}, we can approximate $A\cup R_i$ from the inside by a sequence of smooth open sets $(A_n)_{n\in\N}$ such that $\hs^2(\partial A_n)\to \hs^2(\partial^*(A\cup R_i))$. By Remark \ref{rem:traces}, we have 
$$
\lim_{n\to+\infty}\int_{\partial A_n}(\n_A\cdot \normal_n)^2\,d\hs^2=\int_{\partial^*(A\cup R_i)}(\n_A\cdot \normal)^2\,d\hs^2
$$
where $\normal_n$ denotes the outer unit normal to $A_n$. Consequently, we obtain
$$
\limsup_{n\to +\infty}\free(A_n)< \free(A)+\hat e_\eps
$$
By choosing $\eta_n$ such that $|\eta_nA_n|=m$, we can find a smooth set 
$\hat{A}_\eps:=\eta_nA_n\subset D$ with $|\hat{A}_\eps|=m$ (the explicit dependence on $\varepsilon$ refers to the range within which the volume of $A$ is contained) and
$$
\free(\hat A_\eps)< \tilde\free(\Om,\Gamma)+\eps+\hat e_\eps.
$$
Taking the limit as $\eps\to 0$, we obtain
$$
\inf_{\As_m^{reg}(D)}\free\le \inf_{\As_m(D)}\tilde\free.
$$
Since the reverse inequality is trivial, it follows that \eqref{eq:reg-eq} holds.
\end{proof}

\begin{remark}
\label{rem:box}
{\rm
The convexity assumption for the design region $D$ can be replaced by a smoothness condition. Specifically, instead of the contraction with respect to the origin used in the proof, we can employ the flow of a smooth vector field that points inward at every point on the boundary of $D$ and has negative divergence inside $D$ (ensuring that the flow reduces the volume of $D$).
}
\end{remark}

\begin{remark}
\label{rem:dim-due}
{\rm
The existence of stable configurations, whether in the whole space or within a box, can be established in two dimensions as well. This is because the technical results presented in Section \ref{sec:tech}, as previously noted, apply to any dimension. However, when dealing with a box, proving the density results \eqref{eq:reg-eq} for smooth sets with a prescribed volume $m$ encounters a technical issue related to the extension of an element of $W^{1,2}(A;\mathbb S^1)$ to $W^{1,2}(D;\mathbb S^1)$, where $A\subset\subset D$ is open and smooth. This extension is not always possible because $p=2$ is the critical exponent in two dimensions. Nevertheless, the arguments presented above show that a weaker form of the density result is valid:
$$
\min_{\As_m(D)}\tilde\free=\lim_{\eps\to 0^+} \inf_{\As_{m-\eps}^{reg}(D)}\free.
$$
}
\end{remark}

\section{Tangential boundary conditions}
\label{sec:tan}
In this section, we show that the techniques from the previous sections can be applied to address tangential boundary condition for the directors, i.e. $\n\cdot \normal=0$.
\par
Given $(\Om,\Gamma)\in \As(\R^3)$, define
$$
\Nstg(\Om,\Gamma):=\{\n \in \Ns(\Om,\Gamma)\,:\, \text{$\n\cdot \normal=0$ on $\partial^*\Om$ and $\n^\pm\cdot \normal=0$ on $\Gamma$}\}.
$$
If $\Nstg(\Om,\Gamma)\not=\emptyset$, then we will say that $(\Om,\Gamma)\in \Astg(\R^3)$.
\par
For directors $\n\in \Nstg(\Om,\Gamma)$ the surface free energy reduces to
$$
\tilde\free_s[(\Om,\Gamma), \n]=\gamma[\hdue(\partial^*\Om)+2\hdue(\Gamma)],
$$
indicating that it is purely geometric since the terms depending on the traces of $\n$ vanish on $\partial^*\Om\cup \Gamma$.
\par
We define the total energy of the shape $(\Om,\Gamma)\in \Astg(\R^3)$ as
\begin{equation}
\label{eq:Etottg}
\tilde\freetg(\Om,\Gamma):=\min_{\n\in \Nstg(\Om,\Gamma)}\tilde\free\left[(\Omega,\Gamma),\n\right].
\end{equation}

\begin{theorem}\label{th:th_opt_directortg}
Assume that Ericksen's inequalities \eqref{eq:eineq} hold and let the constant of the surface energy satisfy $\gamma>0$. Then the minimum in \eqref{eq:Etottg} is attained.
\end{theorem}

\begin{proof}
The proof follows the strategy outlined in Theorem \ref{th:th_opt_director}: the strong convergence of the traces on $\partial^*\Om \cup \Gamma$ for a minimizing sequence in $\Nstg(\Om,\Gamma)$ guarantees that the limit function also belongs to $\Nstg(\Om,\Gamma)$. 
\end{proof}

Let $D\subseteq \R^3$ be an open bounded set with $0<m\le |D|$. We use the notation $\As_m(D)$ from Section \ref{sec:mainpf}
to denote the family of configurations in $\As(\R^3)$ that are contained in $D$ and have volume $m$. Similarly, let $\Astg_m(D)$ denote the subfamily of $\As_m(D)$ consisting of the admissible configurations that satisfy tangential boundary conditions.
\par
The minimization of $\tilde\freetg$ over $\Astg_m(D)$ is well posed and can be obtain as the limiting case of the minimization of the standard free energy over $\As_m(D)$ when the parameter $\lambda$ of the surface energy in \eqref{eq:gamma_a} approaches $+\infty$. To state this result properly, let us denote the total free energy of a configuration in $\As_m(D)$ by $\tilde \free^{\lambda}$, explicitly indicating its dependence on $\lambda$. 

\begin{theorem}
\label{thm:maintg}
Let $D\subset\R^3$ be an open bounded set, and let $0<m< |D|$. Assume that Ericksen's inequalities \eqref{eq:eineq} hold and let the constant of the surface energy satisfy $\gamma>0$. Furthermore, let $\Astg_m(D)\not=\emptyset$. Then the minimum problem
\begin{equation}
\label{eq:mainpbtg}
\min_{(\Om,\Gamma)\in \As_m^{tg}(D)} \tilde\freetg(\Om,\Gamma)
\end{equation}
is well posed. Moreover,
\begin{equation}
\label{eq:reg-eq-tg}
\min_{\Astg_m(D)} \tilde\freetg=\lim_{\lambda\to +\infty}\inf_{\As_m(D)} \tilde\free^\lambda.
\end{equation}
\end{theorem}

\begin{proof}
The methodology used here is closely aligned with that used in the proof of Theorem \ref{thm:main}. Let $(\Om_k,\Gamma_k)_{k\in\N}$ be a minimizing sequence for the problem in $\Astg_m(D)$. We can easily obtain the bound 
\begin{equation}
\label{eq:boundper2tg}
\hdue(\partial^*\Om_k\cup \Gamma_k)\le C
\end{equation}
for some constant $C$ independent of $k$. Hence, there exists a set $\Om\subset D$ of finite perimeter such that, up to a subsequence,
\begin{equation*}
\label{eq:Omega_kconvergencetg}
1_{\Om_k}\to 1_\Om\qquad\text{strongly in }L^1(\R^3).
\end{equation*}
Let $\Gamma\subset \Om^1$ be given by Theorem \ref{thm:lsc}, so that $(\Om,\Gamma)\in \As_m(D)$. 
According to Theorem \ref{th:th_opt_directortg}, let $\n_k\in \Nstg(\Om_k,\Gamma_k)$ be such that
$$
\tilde\free(\Om_k,\Gamma_k)=\tilde\free[(\Om_k,\Gamma_k),\n_k].
$$
Using the estimates on the bulk energy given by Remark \ref{rem:bulk} and the bound \eqref{eq:boundper2tg} we obtain $\|\nabla \n_k\|_2^2\le C$ for some $C>0$ independent of $k$. Let $\n\in \ns(\Om,\Gamma)$ be the limit function of $(\n_k)_{k\in\N}$ (up to subsequences) given by Theorem \ref{thm:lsc}. By lower semicontinuity of bulk energies (see \cite[Theorem 5.12]{AFP}) we have
\begin{multline*}
\int_{\Om}W(\n,\nabla \n)\,dx=\int_{D}W(\n,\nabla \n)\,dx\\
\le \liminf_{k\to+\infty}\int_{D}W(\n_k,\nabla \n_k)\,dx=\liminf_{k\to+\infty}\int_{\Om_k}W(\n_k,\nabla \n_k)\,dx.
\end{multline*}
By Lemma \ref{lem:glambda} and the lower semicontinuity result in Theorem \ref{thm:lsc} for $\lambda=0$, we obtain
\begin{equation*}
\hdue(\partial^*\Om)+2\hdue(\Gamma)\le \liminf_{k\to+\infty}
\left[
\hdue(\partial^*\Om_k)+2\hdue(\Gamma_k)
\right].
\end{equation*}
To prove that $(\Om,\Gamma)$ is a effectively a minimum for \eqref{eq:mainpbtg}, it suffices to show that $\n\in \Nstg(\Om,\Gamma)$, which implies that $(\Om,\Gamma)\in \Astg_m(D)$. Applying again the lower semicontinuity result from Theorem \ref{thm:lsc} for $\lambda=0$ to the function $\varphi(\bm v,\bm p):=|\bm v\cdot \bm p|$ we deduce that
\begin{multline*}
\int_{\partial^*\Om}|\n\cdot \normal|\,d\hdue+\int_{\Gamma}|[\n^+\cdot \normal|+|\n^-\cdot\normal|]\,d\hdue\\
\le
\liminf_{k\to+\infty}
\left[\int_{\partial^*\Om_k}|\n_k\cdot \normal_k|\,d\hdue+\int_{\Gamma_k}|[\n_k^+\cdot \normal_k|+|\n_k^-\cdot\normal_k|]\,d\hdue\right]=0.
\end{multline*}
This inequality yields $\n\cdot \normal=0$ and $\n^\pm\cdot \normal=0$ $\hdue$-a.e. on $\partial^*\Om$ and $\Gamma$, respectively. Therefore, $\n\in \Nstg(\Om,\Gamma)$.
\par
Let us address \eqref{eq:reg-eq-tg}. Since $\tilde\free^\lambda(\Om,\Gamma)\le \tilde\freetg(\Om,\Gamma)$ for $(\Om,\Gamma)\in \Astg_m(D)$, we have
$$
\limsup_{\lambda\to +\infty}\inf_{\As_m(D)} \tilde\free^\lambda\le \min_{\Astg_m(D)} \tilde\freetg.
$$
To complete the proof, it suffices to show that
\begin{equation}
\label{eq:claim-inf}
\min_{\Astg_m(D)} \tilde\freetg\le \liminf_{\lambda\to +\infty}\inf_{\As_m(D)} \tilde\free^\lambda.
\end{equation}
Consider a monotone sequence $(\lambda_k)_{k\in\N}$ with $\lambda_k\to +\infty$, and a sequence $(\Om_k,\Gamma_k)_{k\in\N}$ in $\As_m(D)$ such that
$$
\tilde\free^{\lambda}(\Om_k,\Gamma_k)\le \inf_{\As_m(D)} \tilde\free^\lambda+\frac{1}{k}\le \min_{\Astg_m(D)} \tilde\freetg+\frac{1}{k}.
$$
Notice that
$$
\hdue(\partial^*\Om_k)+\hdue(\Gamma_k)\le C
$$
for some constant $C$ independent of $k$.
Let $(\Om,\Gamma)\in \As_m(D)$ be the limit configuration, up to a subsequence, associated to $(\Om_k,\Gamma_k)$ according to Theorem \ref{thm:lsc}.
Let $\n_k\in \Ns(\Om_k,\Gamma_k)$ be a director which realizes $\tilde\free^{\lambda_k}(\Om_k,\Gamma_k)$ (see Theorem \ref{th:th_opt_director}), and let $\n\in \ns(\Om,\Gamma)$ be the associated limit, up to a subsequence. We claim that
\begin{equation}
\label{eq:claimtg}
\n\in \Nstg(\Om,\Gamma).
\end{equation}
Indeed, by the lower semicontinuity result in Theorem \ref{thm:lsc}, \eqref{eq:boundProb}, we have for every $\alpha>0$,
\begin{multline*}
\alpha\left[\int_{\partial^*\Om}|\n\cdot \normal|\,d\hdue+\int_{\Gamma}|[\n^+\cdot \normal|+|\n^-\cdot\normal|]\,d\hdue\right]\\
\le
\liminf_{k\to+\infty}
\left[\int_{\partial^*\Om_k}\alpha|\n_k\cdot \normal_k|\,d\hdue+\int_{\Gamma_k}\alpha|[\n_k^+\cdot \normal_k|+|\n_k^-\cdot\normal_k|]\,d\hdue\right]\\
\le
\liminf_{k\to+\infty}
\left[\int_{\partial^*\Om_k}[1+\lambda_k(\n_k\cdot \normal_k)^2]\,d\hdue+\int_{\Gamma_k}[2+\lambda_k(\n_k^+\cdot \normal_k)^2+\lambda_k(\n_k^-\cdot\normal_k)^2]\,d\hdue\right]\le C,
\end{multline*}
so that, letting $\alpha\to+\infty$, we deduce
$$
\int_{\partial^*\Om}|\n\cdot \normal|\,d\hdue+\int_{\Gamma}|[\n^+\cdot \normal|+|\n^-\cdot\normal|]\,d\hdue=0,
$$
and claim \eqref{eq:claimtg} follows. Consequently, $(\Om,\Gamma)\in \Astg_m(D)$. By the lower semicontinuity of quasiconvex bulk energies in SBV (c.f. \cite[Theorem  5.29]{AFP}) applied our $\free_\mathrm{b}$ in \eqref{eq:Eb}, we obtain
$$
\int_{\Om}W(\n,\nabla \n)\,dx\le \liminf_{k\to+\infty}\int_{\Om_k}W(\n_k,\nabla \n_k)\,dx.
$$
Additionally, by Lemma 5.1 and applying the lower semicontinuity result from Theorem \ref{thm:lsc} for $\lambda=0$ to the function $\varphi(\bm v,\bm p):=|\bm v\cdot \bm p|$, we have
$$
\hdue(\partial^*\Om)+2\hdue(\Gamma)\le
\liminf_{k\to+\infty}\left[\hdue(\partial^*\Om_k)+\hdue(\Gamma_k)\right].
$$
Hence, we deduce
$$
\tilde\freetg(\Om,\Gamma)\le \liminf_{k\to+\infty}\tilde\free^{\lambda_k}(\Om_k,\Gamma_k)
$$
which implies that \eqref{eq:claim-inf} follows.
\end{proof}

\begin{remark}
\label{rem:lambda-tg}
{\rm
The approximation through smooth sets of the domain $\Omega$ proves difficult under tangential boundary conditions, because the methods from Section \ref{sec:density} cannot be directly adapted to handle the constraint $\n\cdot\normal=0$ on the boundary. However, the density result in \eqref{eq:reg-eq} for $\tilde\free^\lambda$ is preserved in the limit as $\lambda\to +\infty$.
}
\end{remark}

\paragraph*{\bf Acknowledgements}  
A.G. acknowledges support by PRIN2022 n. 2022J4FYNJ funded by MUR, Italy, and by the European Union - Next Generation EU.
He is also member of the Gruppo Nazionale per l’Analisi Matematica, la Probabilit\`a e le loro Applicazioni (GNAMPA) of the Istituto Nazionale di Alta Matematica (INdAM). \\
S.P. is member of the Italian Gruppo Nazionale per la Fisica Matematica, which is part of INdAM.
\par
This manuscript has no associated data.

\bibliographystyle{plain}
\bibliography{agbiblio}

\end{document}